\documentclass[12,ptoneside]{amsart}
\usepackage[letterpaper, margin=1.25in]{geometry}
\usepackage{bm}
\usepackage{xcolor}
\usepackage{comment}

\makeatletter

\usepackage[numbers]{natbib}
\usepackage{appendix}

\usepackage{enumerate}
\usepackage[ruled,linesnumbered]{algorithm2e}
\usepackage{multirow}
\usepackage{subcaption}
\usepackage{tabularx}
\usepackage{makecell}
\usepackage{booktabs}

\usepackage{amsmath}
\usepackage{amstext}
\usepackage{amsthm}
\usepackage{amssymb}
\usepackage{bbm}
\usepackage{amsbsy}
\usepackage{mathtools}

    \newtheorem{prop}{Proposition}[section]
    \newtheorem{lem}{Lemma}[section]
    \newtheorem{rem}{Remark}[section]
    \newtheorem{cor}{Corollary}[section]
    \newtheorem{thm}{Theorem}[section]
    \newtheorem{defn}{Definition}
    
    \providecommand{\examplename}{Example}
        \newtheorem{example}{\protect\examplename}
        \theoremstyle{plain}
    \theoremstyle{plain}
    \newtheorem*{thm*}{Theorem}
        
    \def\bx {\boldsymbol{x}}
    \def\by {\boldsymbol{y}}
    \def\bu {\boldsymbol{u}}
    \def\bv {\boldsymbol{v}}
    \def\bp {\boldsymbol{p}}
    \def\bh {\boldsymbol{h}}
    \def\bS {\boldsymbol{S}}

    \def\R {\mathbb{R}}
    
    \def\dom {\mathrm{dom~}}
    \def\inter {\mathrm{int~}}
    
    \def\ri {\mathrm{ri~}}
    \def\cl {\mathrm{cl~}}
    \def\bd {\mathrm{bd~}}
    \def\Rn {\R^n}
    \def\gmRn {\Gamma_0(\Rn)}
    \def\Laplacian{\Delta_{\boldsymbol x}}
    
    \def\erfc {\rm erfc}
    \def\cballr {B_r(\boldsymbol{0})}
    \newcommand{\diff}{\mathop{}\!d} 

    \DeclareMathOperator*{\argmin}{arg\,min}
    \DeclareMathOperator*{\argmax}{arg\,max}
    \DeclareMathOperator*{\essinf}{ess\,inf}

    \newcommand{\expectationJ}[1]{\mathbb{E}_{J}\left [{#1}\right] }
    \newcommand{\normsq}[1]{\left\Vert{#1} \right\Vert_{2}^{2}}

\usepackage[backref=page]{hyperref}
\renewcommand*{\backref}[1]{}
\renewcommand*{\backrefalt}[4]{%
    \ifcase #1 (Not cited.)%
    \or        (Cited on page~#2.)%
    \else      (Cited on pages~#2.)%
    \fi}

\begin{document}

\title[Bayesian estimators in imaging sciences and HJ PDEs]{On Bayesian posterior mean estimators in imaging sciences and Hamilton-Jacobi Partial Differential Equations}

\date{Research supported by NSF DMS 1820821. Authors' names are given in last/family name alphabetical order.}

\author[J. Darbon]{J\'er\^ome Darbon}
\address{Division of Applied Mathematics, Brown University.}
\email{jerome\_darbon@brown.edu}

\author[G. P. Langlois]{Gabriel P. Langlois}
\address{Division of Applied Mathematics, Brown University.}
\email{gabriel\_provencher\_langlois@brown.edu}

\maketitle


\begin{abstract}
Variational and Bayesian methods are two approaches that have been widely used to solve image reconstruction problems. In this paper, we propose original connections between Hamilton--Jacobi (HJ) partial differential equations and a broad class of Bayesian methods and posterior mean estimators with Gaussian data fidelity term and log-concave prior. Whereas solutions to certain first-order HJ PDEs with initial data describe maximum a posteriori estimators in a Bayesian setting, here we show that solutions to some viscous HJ PDEs with initial data describe a broad class of posterior mean estimators. These connections allow us to establish several representation formulas and optimal bounds involving the posterior mean estimate. In particular, we use these connections to HJ PDEs to show that some Bayesian posterior mean estimators can be expressed as proximal mappings of twice continuously differentiable functions, and furthermore we derive a representation formula for these functions.
\end{abstract}

\section{Introduction} \label{sec:intro}
Image denoising problems consist in estimating an unknown image from a noisy observation in a way that accounts for the underlying uncertainties, and variational and Bayesian methods have become two important approaches for doing so. The goal of this paper is to describe a broad class of Bayesian posterior mean estimators with Gaussian data fidelity term and log-concave prior using Hamilton--Jacobi (HJ) partial differential equations (PDEs) and use these connections to clarify certain image denoising properties of this class of Bayesian posterior estimators.

To illustrate the main ideas of this paper, we first briefly introduce convex finite-dimensional variational and Bayesian methods relevant to image denoising problems. Variational methods formulate image denoising problems as the optimization of a weighted sum of a data fidelity term (which embeds the knowledge of the nature of the noise corrupting the unknown image) and a regularization term (which embeds known properties of the image to reconstruct) \cite{chambolle2016introduction,darbon2015convex}, where the goal is to minimize this sum to obtain an estimate that hopefully accounts well for both the data fidelity term and the regularization term. Bayesian methods formulate image denoising problems in a probabilistic framework that combine observed data through a likelihood function (which models the noise corrupting the unknown image) and prior knowledge through a prior distribution (which models known properties of the unknown image) to generate a posterior distribution. An appropriate decision rule that minimizes the posterior expected value of a loss function, also called a Bayes estimator, then selects a meaningful image estimate from the posterior distribution that hopefully accounts well for both the prior knowledge and observed data \cite{demoment1989image,stuart2010inverse,tarantola2005inverse,tikhonov1995numerical,vogel2002computational}. A standard example is the posterior mean estimator, which minimizes the posterior expected value of the square error from the noisy observation, and it corresponds to the mean of the posterior distribution (\cite{kay1993fundamentals}, pages 344-345).

We will focus on the following class of variational and Bayesian imaging models: given an observed image $\bx \in \Rn$ corrupted by additive Gaussian noise and parameters $t > 0$ and $\epsilon > 0$, estimate the original uncorrupted image by computing, respectively, the maximum a posteriori (MAP) and posterior mean (PM) estimates:
\begin{equation} \label{eq:minimizer_prob}  
\bu_{MAP}(\bx,t) \coloneqq \argmin_{\by\in\Rn}\left\{ \frac{1}{2t}\normsq{\bx-\by}+J(\by)\right\}
\end{equation}
and
\begin{equation} \label{eq:pm_estimate}
    \bu_{PM}(\bx,t,\epsilon) \coloneqq \frac{\int_{\Rn} \by e^{-\left(\frac{1}{2t}\normsq{\bx-\by} + J(\by)\right)/\epsilon} \diff\by}{\int_{\Rn} e^{-\left(\frac{1}{2t}\normsq{\bx-\by} + J(\by)\right)/\epsilon} \diff\by}.
\end{equation}
The functions $\by\mapsto \frac{1}{2t}\normsq{\bx-\by}$ and $J\colon\Rn\to\R\cup\{+\infty\}$ in \eqref{eq:minimizer_prob} are, respectively, the (Gaussian) data fidelity and regularization terms, and the functions $\by \mapsto e^{-\left(\frac{1}{2t}\normsq{\bx-\by} + J(\by)\right)/\epsilon}$ and $\by \mapsto e^{-J(\by)/\epsilon}$ in \eqref{eq:pm_estimate} are, respectively, the (Gaussian) likelihood function and generalized prior distribution. The parameter $t>0$ controls the relative importance of the data fidelity term over the regularization term, and the parameter $\epsilon$ controls the shape of the posterior distribution in \eqref{eq:pm_estimate}, where small values of $\epsilon$ favor configurations close to the mode, which is the MAP estimate, of the posterior distribution.

To illustrate the MAP and PM estimates and their denoising capabilities, we give an example based on the Rudin--Osher--Fatemi (ROF) image denoising model, which consists of a total variation (TV) regularization term with quadratic data fidelity term \cite{bouman1993generalized,chambolle1997image,rudin1992nonlinear}. We assume that images are defined on a lattice $\mathcal{V}$ of cardinality $|\mathcal{V}| = n$. The value of an image $\bx$ at a size $i \in \mathcal{V}$ is denoted by $x_i \in \R$. We consider specifically the finite dimension anisotropic TV term endowed with 4-nearest neighbors interactions \cite{winkler2003image}, which takes the form
\begin{equation*}
    TV(\by) = \frac{1}{2}\sum_{i \in \mathcal{V}}\sum_{j \in \mathcal{N}(i)}|y_i-y_j|.
\end{equation*}
The associated anisotropic ROF problem \cite{rudin1992nonlinear} takes the form
\begin{equation}\label{eq:rof_model}
\inf_{y_i \in \mathbb{R} \colon i \in \mathcal{V}}\left\{\sum_{i \in \mathcal{V}} \frac{1}{2t}(x_i - y_i)^2 + \frac{1}{2}\sum_{i \in \mathcal{V}}\sum_{j \in \mathcal{N}(i)}|y_i-y_j| \right\} \equiv \inf_{\by \in \Rn}\left\{\frac{1}{2t}\normsq{\bx-\by} + TV(\by) \right\}.
\end{equation}
In a Bayesian setting, the posterior mean estimate associated to the anisotropic ROF problem above is
\begin{equation}\label{eq:pm_tv}
\bu_{PM}(\bx,t,\epsilon) = \frac{\int_{\Rn} \by e^{-\left(\frac{1}{2t}\normsq{\bx-\by} + TV(\by)\right)/\epsilon} \diff \by}{\int_{\Rn} e^{-\left(\frac{1}{2t}\normsq{\bx-\by} + TV(\by)\right)/\epsilon} \diff\by}.
\end{equation}
Figure~\ref{fig:original_images}(A) depicts the image Barbara, which we corrupt with Gaussian noise (zero mean with standard deviation $\sigma = 20$ with pixel values in $[0,255]$) in Figure~\ref{fig:original_images}(B). The resultant noisy image $\bx$ is denoised by computing the minimizer $\bu_{MAP}(\bx,t)$ to the ROF model~\eqref{eq:rof_model} with $t = 20$ and is illustrated in Figure~\ref{fig:original_images}(C), and the posterior mean estimate $\bu_{PM}(\bx,t,\epsilon)$~\eqref{eq:pm_tv} associated to the ROF model~\eqref{eq:rof_model} with $t = \epsilon = 20$ and illustrated in Figure~\ref{fig:original_images}. Figure 1(C) illustrates the denoised image using the minimizer to the TV problem~\eqref{eq:rof_model}, and Figure 1(D) illustrates the denoised image using the posterior mean estimate~\eqref{eq:pm_tv}.

\begin{figure}[ht]
    \centering
\begin{subfigure}{0.55\textwidth}
    \centering
    \includegraphics[width = 0.60\textwidth]{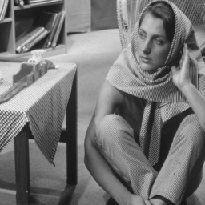}
    \caption{}
    \label{subfig:lena}
\end{subfigure}%
\begin{subfigure}{0.55\textwidth}
    \centering
    \includegraphics[width = 0.60\textwidth]{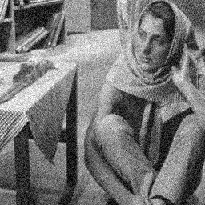}
    \caption{}
    \label{subfig:lena_noisy}
\end{subfigure}

\begin{subfigure}{0.55\textwidth}
    \centering
    \includegraphics[width = 0.60\textwidth]{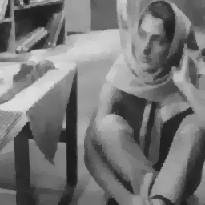}
    \caption{}
    \label{subfig:lena_denoised_umap}
\end{subfigure}%
\begin{subfigure}{0.55\textwidth}
    \centering
    \includegraphics[width = 0.60\textwidth]{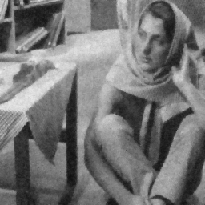}
    \caption{}
    \label{subfig:lena_denoised_upm}
\end{subfigure}
    \caption{The anisotropic ROF model for image denoising applied to the noisy grayscale 256x256 image of Barbara.  The noise is Gaussian with zero mean and standard deviation $\sigma = 20$, and the parameters of the ROF model~\eqref{eq:rof_model} and posterior mean estimate~\eqref{eq:pm_tv} are $t=20$ and $\epsilon=20$. (A) Original image of Barbara. (B) Noisy image of Barbara. (C) Denoised image of Barbara with the MAP estimate. (D) Denoised image of Barbara with PM estimate.}
    \label{fig:original_images}
\end{figure}

Variational methods are popular because the resultant optimization problem for various non-smooth and convex regularization terms used in image denoising problems, such as total variation and $l_1$-norm based regularization terms, is well-understood \cite{bouman1993generalized,candes2006robust,chambolle1997image,darbon2015convex,darbon2019decomposition,donoho2006compressed,rudin1992nonlinear} and can be solved efficiently using robust numerical optimization methods \cite{chambolle2016introduction}. These methods may have shortcomings, however, in that reconstructed images from variational methods with non-smooth and convex regularization terms may have undesirable and visually unpleasant staircasing effects due to the singularities of the non-smooth regularization terms \cite{chan2000high,dobson1996recovery,durand2000image,nikolova2007model,louchet2008modeles,woodford2009global}. This is illustrated in the example above in Figure 1(C), which contains regions where the pixel values are equal and lead to staircasing effects. While posterior mean estimates are typically slower to compute than MAP estimates, posterior mean estimates with Gaussian fidelity term and total variation regularization terms have been shown to avoid staircasing effects \cite{louchet2008modeles,louchet2013posterior}. This is illustrated for example in Figure 1(D), where the denoised image with posterior mean estimate does not contain visibly substantial regions where the pixel values are equal.

Several papers have proposed original connections between MAP and Bayesian estimators, including posterior mean estimators. First, the papers of \cite{louchet2008modeles,gribonval2011should,gribonval2013reconciling,louchet2013posterior} showed that the class of Bayesian posterior mean estimates \eqref{eq:pm_estimate} can be expressed as minimizers to optimization problems involving a Gaussian fidelity term and a smooth convex regularization term, i.e., there exists a smooth regularization term $f_{\rm{reg}}\colon \Rn \to \R$ such that
\begin{equation}\label{eq:pm_rep_map}
\bu_{PM}(\bx,t,\epsilon) = \argmin_{\by \in \Rn} \left\{\frac{1}{2}\normsq{\bx-\by} + f_{\rm{reg}}(\by)\right\}.
\end{equation}
This result was later extended to some non-Gaussian data fidelity terms in \cite{gribonval2018characterization,gribonval2018bayesian}. These results, however, proved existence of the regularization term $f_{\rm{reg}}$ and not their explicit form. Second, the papers of \cite{burger2014maximum,burger2016bregman,pereyra2019revisiting} showed that the MAP estimate~\eqref{eq:minimizer_prob} can, under certain assumptions on the regularization term $J$, be characterized as a proper Bayes estimator, that is, the MAP estimate~\eqref{eq:minimizer_prob} minimizes the posterior expected value of an appropriate loss function.

In addition to these results, it is known that under certain assumptions on the regularization term $J$, the value of the minimization problem
\begin{equation} \label{eq:minimization_prob} 
S_{0}(\bx,t)\coloneqq\min_{\by\in\Rn}\left\{ \frac{1}{2t}\normsq{\bx-\by}+J(\by)\right\}
\end{equation}
whose minimizer is the MAP estimate~\eqref{eq:minimizer_prob}, satisfies the first-order HJ PDE 
\begin{equation}
\begin{dcases}
\frac{\partial S_0}{\partial t}(\bx,t)+\frac{1}{2}\normsq{\nabla_{\bx}S_0(\bx,t)}=0 & \mbox{in }\Rn\times(0,+\infty),\\
S_0(\bx,0)=J(\bx) & \mbox{in }\Rn.
\end{dcases}
\end{equation}
The properties of the minimizer $\bu_{MAP}(\bx,t)$ follow from the properties of the solution to this HJ equation~\cite{hiriart2013convexII,moreau1965proximite,rockafellar2009variational,darbon2015convex}. In particular, the MAP estimate satisfies the representation formula $\bu_{PM}(\bx,t) = \bx - t\nabla_{\bx}S_0(\bx,t)$. Similar results have been established for finite-dimensional variational models in imaging sciences \cite{darbon2019decomposition}.

The connections between Bayesian posterior mean estimates and MAP estimates as identified by \cite{louchet2008modeles,gribonval2011should,gribonval2013reconciling,louchet2013posterior,gribonval2018characterization,gribonval2018bayesian}, the connections between MAP estimates and Bayesian estimators as identified by \cite{burger2014maximum,burger2016bregman,pereyra2019revisiting}, and the connections between MAP estimates and HJ PDEs suggest there may be deep connections between Bayesian estimators, including posterior mean estimates, and HJ PDEs. This paper proposes to establish original connections between Bayesian posterior mean estimates and HJ PDEs. We shall see in this paper that under appropriate conditions on the regularization term $J$, the posterior mean estimate~\eqref{eq:pm_estimate} is described by the solution to a viscous HJ PDE with initial data $J$. Specifically, the function $S_\epsilon\colon\Rn\times[0,+\infty)\to[0,+\infty)$ defined by
\begin{equation}\label{eq:intro_Seps}
S_{\epsilon}(\bx,t) \coloneqq -\epsilon\ln\left(\frac{1}{\left(2\pi t\epsilon\right)^{n/2}}\int_{\Rn}e^{-\left(\frac{1}{2t}\left\Vert \bx-\by\right\Vert _{2}^{2}+J(\by)\right)/\epsilon}\diff\by\right)
\end{equation}
solves the viscous HJ equation with initial data
\begin{equation}
\begin{dcases}\label{intro:pde}
\frac{\partial S_{\epsilon}}{\partial t}(\bx,t)+\frac{1}{2}\left\Vert \nabla_{\bx}S_{\epsilon}(\bx,t)\right\Vert _{2}^{2}=\frac{\epsilon}{2}\Laplacian S_{\epsilon}(\bx,t) & \text{ in }\Rn\times(0,+\infty),\\
S_{\epsilon}(\bx,0)=J(\bx) & \text{ in }\Rn,
\end{dcases}
\end{equation}
and the PM estimate satisfies the representation formula
\begin{equation}
\bu_{PM}(\bx,t,\epsilon)=\bx-t\nabla_{\bx}S_{\epsilon}(\bx,t).
\end{equation}
Moreover, we shall see that the connections between the posterior mean estimate~\eqref{eq:pm_estimate} and viscous HJ PDE~\eqref{intro:pde} will allow us to find the exact representation of the regularization term $f_{\rm{reg}}$ in \eqref{eq:pm_rep_map}. These connections will also enable us to show that when the regularization function $J$ is convex on $\Rn$, with no further regularity such as continuous differentiability or uniform Lipschitz continuity, then the MAP estimate is also a proper Bayes estimator, i.e., it minimizes the posterior expected value of a loss function.

\subsection{Contributions} \label{subsec:intro_contributions}
In this paper, we propose original connections between solutions to HJ PDEs and a broad class of Bayesian methods and posterior mean estimators. These connections are described in Theorems~\ref{thm:visc_hj_eqns} and~\ref{thm:char_opt_sol} for viscous HJ PDEs and first-order HJ PDEs, respectively. We show in Theorem~\ref{thm:visc_hj_eqns} that the posterior mean estimate $\bu_{PM}(\bx,t,\epsilon)$ is described by the solution to the viscous HJ PDE~\eqref{intro:pde} with initial data corresponding to the convex regularization term $J$, which we characterize in detail in terms of the data $\bx$ and parameters $t$ and $\epsilon$. In particular, the posterior mean estimate satisfies the representation formula $\bu_{PM}(\bx,t,\epsilon) = \bx - t\nabla_{\bx}S_\epsilon(\bx,t)$. Next, we use the connections between viscous HJ PDEs and posterior mean estimates established in Theorem~\ref{thm:visc_hj_eqns} to show in Theorem~\ref{thm:char_opt_sol} that the posterior mean estimate $\bu_{PM}(\bx,t,\epsilon)$ can be expressed through the gradient of the solution to a first-order HJ PDE with twice continuously differentiable convex initial data $\Rn \ni \bx \mapsto K^{*}_{\epsilon}(\bx,t)-\frac{1}{2}\Vert \bx \Vert_{2}^{2}$, where
\[
K_{\epsilon}(\bx,t) = t\epsilon\ln\left(\frac{1}{(2\pi t\epsilon)^{n/2}}\int_{\dom J}e^{\left(\frac{1}{t}\left\langle \bx,\by\right\rangle -\frac{1}{2t}\left\Vert \by\right\Vert _{2}^{2}-J(\by)\right)/\epsilon}\diff\by\right)
\]
and $\bx \mapsto  K^{*}_{\epsilon}(\bx,t)$ is the Fenchel--Legendre transform of the function $\bx \mapsto K_{\epsilon}(\bx,t)$. In other words, we show
\[
\bu_{PM}(\bx,t,\epsilon) = \argmin_{\by \in \Rn}\left\{\frac{1}{2}\normsq{\bx-\by} + \left( K^{*}_{\epsilon}(\by,t)-\frac{1}{2}\Vert \by \Vert_{2}^{2}\right)\right\}.
\]
This formula gives the representation of the convex regularization term enabling one to express the posterior mean estimate as the minimizer of a convex variational problem, and in fact in terms of the solution to a first-order HJ PDE, thereby extending the results of \cite{louchet2008modeles,gribonval2011should} who showed existence of this regularization term when the data fidelity term is Gaussian, but not its representation. The twice continuously differentiability of this regularization term, in particular, implies that the posterior mean estimate $\bu_{PM}(\bx,t,\epsilon)$ avoids image denoising staircasing effects as a consequence of the results derived by \citet{nikolova2004weakly} (specifically, by Theorem 3 in her paper).

We also use the connections between posterior mean estimators and solutions to viscous HJ PDEs established in Theorem~\ref{thm:visc_hj_eqns} to prove several topological, representation, and monotonicity properties of posterior mean estimators, respectively, in Propositions~\ref{prop:topo_properties}, \ref{prop:pm_rep_props}, and \ref{prop:mono_properties}. These properties are then used in Proposition~\ref{prop:bounds_limits} to derive an optimal upper bound on the mean squared error (MSE) error $\expectationJ{\normsq{\by-\bu_{PM}(\bx,t,\epsilon)}}$, several estimates on the MAP and posterior mean estimates, and the behavior of the posterior mean estimate $\bu_{PM}(\bx,t,\epsilon)$ in the limit $t \to 0$. 
Finally, we use the connections between both MAP and posterior mean estimates and HJ PDEs to characterize the MAP estimate~\eqref{eq:minimizer_prob} in the context of Bayesian estimation theory, and specifically in Theorem~\ref{thm:bregman_div} to show that the MAP estimate~\eqref{eq:minimizer_prob} corresponds to the Bayes estimator of the Bayesian risk~\eqref{eq:breg2} whenever $J$ is convex on $\Rn$. Under the assumption that the data fidelity term is Gaussian, our result extends the findings of \cite{burger2014maximum} and \cite{pereyra2019revisiting} by removing the restriction on $J$ to be uniformly Lipschitz continuous on $\Rn$. When $J$ is not defined everywhere on $\Rn$, we show that the Bayesian risk~\eqref{eq:breg2} has a corresponding Bayes estimator that is described in terms of the solution to both the first-order HJ PDE~\eqref{thm:e_u_nonvisc_hj} and the viscous HJ PDE~\eqref{thm:visc_hj_eqns}.

\subsection{Organization}\label{subsec:intro_organization}
In Section \ref{sec:background}, we review concepts of real and convex analysis that will be used throughout this paper. In Section~\ref{sec:connections}, we establish theoretical connections between a broad class of Bayesian posterior mean estimators and HJ PDEs. Our mathematical set-up is described in Subsection~\ref{subsec:set-up}, the connections of posterior mean estimators to viscous HJ PDEs are described in Subsection~\ref{subsec:HJ_eqns_visc}, and the connections of posterior mean estimators to first-order HJ PDEs are described in Subsection~\ref{subsec:HJ_eqns_1order}. We use these connections to establish various properties of posterior mean estimators in Section~\ref{sec:properties_estimators}. Specifically, we establish topological, representation, and monotonicity properties of posterior mean estimators in Subsection~\ref{subsec:prop_topo-rep-mono}, an optimal upper bound for the MSE, various estimates and bounds involving the posterior mean estimate, and the behavior of the posterior mean estimate in the limit $t \to 0$ in Subsection~\ref{subsec:prop_bounds-est}. Finally, we establish properties of MAP and posterior mean estimators in terms of Bayesian risks involving Bregman divergences in Subsection~\ref{subsec:prop_breg_div}.
\section{Background} \label{sec:background}
This section reviews concepts from real and convex analysis that will be used in this paper. We refer the reader to \cite{folland2013real,hiriart2013convexI,hiriart2013convexII,rockafellar1970convex,rockafellar2009variational} for comprehensive references. In what follows, the Euclidean scalar product on $\Rn$ will be denoted by $\left\langle \cdot,\cdot\right\rangle $ and its associated norm by $\left\Vert \cdot\right\Vert _{2}$. The closure and interior of a non-empty set $C\subset\Rn$ will be denoted by $\cl C$ and $\inter C$, respectively. The boundary of a non-empty set $C\subset\Rn$ is defined as $\cl C \setminus \inter C$ and will be denoted by $\bd C$. The domain of a function $f\colon\Rn\to\R\cup\{+\infty\}$ is the set
\[
\dom f = \left\{ \bx\in\Rn : f(\bx)<+\infty\right\}.
\]
We will denote the Borel $\sigma$-algebra on $\Rn$ by $\mathcal{B}(\Rn)$, and if given a Borel-measurable set $\Omega \subset \mathcal{B}(\Rn)$ and a Lebesgue measurable function $f\colon \Rn \to \R$, we denote the Lebesgue integral of $f$ over $\Omega$ by
\[
\int_{\Omega} f(\by) \diff\by,
\]
where $d\by$ denotes the $n$-dimensional Lebesgue measure.

\begin{defn}[Proper and lower semicontinuous functions] \label{def:domains_prop}
A function $f\colon\Rn\to\R\cup\{+\infty\}$ is proper if $\dom f \neq \varnothing$ and $f(\bx) > -\infty$ for every $\bx \in \dom f$. 

A function $f\colon\Rn\to\R\cup\{+\infty\}$ is lower semicontinuous at $\bx\in\Rn$ if it satisfies $\liminf_{k\to+\infty}f(\bx_{k})\geqslant f(\bx)$ for every sequence $\left\{ \bx_{k}\right\} _{k=1}^{+\infty}\subset\Rn$ such that $\lim_{k\to+\infty}\bx_{k}=\bx$.
\end{defn}
\begin{defn}[Differentiability and the gradient] \label{def:diff_grad}
Let $f\colon \Rn \to \R\cup\{+\infty\}$ be a proper function and let $\bx \in \Rn$ be a point where $f$ is finite. The function $f$ is \textit{differentiable} at $\bx$ if there exists a linear form $Df(\bx) \colon \Rn\to\R$ such that
\[
\lim_{\left\Vert \bh \right\Vert_{2} \to 0}\frac{\left\Vert f(\bx+\bh)-f(\bx)-Df(\bx)(\bh) \right\Vert_{2}}{\left\Vert \bh \right\Vert_{2}} = 0.
\]
The linear form $Df(\bx)$, if it exists, can be represented by a unique vector in $\Rn$ denoted by $\nabla f(\bx)$ through $Df(\bx)(\bh) = \left \langle \nabla f(\bx),\bh \right\rangle$ for every $\bh \in \Rn$. The element $\nabla f(\bx)$ is called the gradient of $f$ at $\bx$.
\end{defn}
\begin{defn}[Convex sets and their relative interiors] \label{def:convex_sets}
A subset $C\subset\Rn$ is convex if for every pair $(\bx,\by) \in C \times C$ and every scalar $\lambda \in (0,1)$, the line segment $\lambda\bx+(1-\lambda)\by \in C$.

The relative interior of a convex set $C$, denoted by $\ri(C)$, is the set of points in the interior of the unique smallest affine set containing $C$. Every convex set $C$ with non-empty interior is $n$-dimensional with $\ri C=\inter C$ and has positive Lebesgue measure, and furthermore the Lebesgue measure of the boundary $\bd C$ equals zero \citep{lang1986note}.
\end{defn}
\begin{defn}[Convex functions and the set $\gmRn$] \label{def:convex} 
A proper function $f\colon\Rn\to\R\cup\{+\infty\}$ is convex if its domain is convex and if for every pair $(\bx,\by) \in \dom f \times \dom f$ and every scalar $\lambda\in(0,1)$, the inequality
\begin{equation}
f(\lambda\bx+(1-\lambda)\by)\leqslant\lambda f(\bx)+(1-\lambda)f(\by)\label{eq:convex_def}
\end{equation}
holds in $\R\cup\{+\infty\}$.

The class of proper, convex and lower semicontinuous functions is denoted by $\gmRn$.

A proper function $f$ is \textit{strictly convex} if the inequality is strict in (\ref{eq:convex_def}) whenever $\bx\neq\by$, and it is \textit{strongly convex with parameter $m>0$} if for every pair $(\bx,\by) \in \dom f \times \dom f$ and every scalar $\lambda\in(0,1)$, the inequality
\[
f(\lambda\bx+(1-\lambda)\by)\leqslant\lambda f(\bx)+(1-\lambda)f(\by)-\frac{m}{2}\lambda(1-\lambda)\left\Vert \bx-\by\right\Vert _{2}^{2}.
\]
holds in $\R\cup\{+\infty\}$. 

A function $g\colon\Rn\to(0,+\infty)$ is log-concave (respectively, \textit{strictly log-concave}, \textit{strongly log-concave of parameter} $m>0$) if the function $-\ln(g)$ is convex (respectively strictly convex, strongly convex of parameter $m>0$).
\end{defn}
\begin{defn}[Projections] \label{def:projections} 
Let $C$ be a closed convex subset of $\Rn$. To every $\bx \in \Rn$, there exists a unique element $\pi_{C}(\bx) \in C$ called the projection of $\bx$ onto $C$ that is closest to $\bx$ in Euclidean norm, i.e.,
\begin{equation} \label{eq:proj_def}
\pi_{C}(\bx)\coloneqq\argmin_{\by\in C}\left\Vert \bx-\by\right\Vert _{2}^{2}.
\end{equation}
This correspondence defines a map $\bx \mapsto \pi_{C}(\bx)$ from $\Rn$ to $C$ called the projector onto $C$ (\cite{aubin2012differential}, Chapter 0.6, Corollary 1). It satisfies the following characterization:
\begin{equation} \label{eq:proj_char}
\left\langle \bx-\pi_{C}(\bx),\by-\pi_{C}(\bx)\right\rangle \leqslant0, \quad\forall\by\in C.
\end{equation}
\end{defn}
\begin{defn}[Subdifferentials and subgradients] \label{def:subgrad} 
Let $f \in \gmRn$. The \textit{subdifferential} of $f$ at $\bx\in\dom f$ is the set $\partial f(\bx)$ of vectors $\bp\in\Rn$ that for every $\by \in \Rn$ satisfies the inequality
\begin{equation} \label{eq:subgrad_def}
f(\by)\geqslant f(\bx)+\left\langle \bp,\by-\bx\right\rangle.
\end{equation}
The vectors $\bp\in\partial f(\bx)$ are called the subgradients of $f$ at $\bx$. The set of points $\bx\in\dom f$ for which the subdifferential $\partial f(\bx)$ is non-empty is denoted by $\dom\partial f$. and it includes the relative interior of the domain of $f$, that is, $\ri(\dom f) \subset \dom\partial f$ (\citep{rockafellar1970convex}, Theorem 23.4). 

If $f$ is strongly convex of parameter $m > 0$, then the subgradients of $f$ at $\bx \in \dom \partial f$ satisfy the stronger inequality
\[
f(\by)\geqslant f(\bx)+\left\langle \bp,\by-\bx\right\rangle + \frac{m}{2}\normsq{\by-\bx}.
\]

If $f$ is differentiable at $\bx$, then $\bx \in \dom\partial f$ and the gradient $\nabla f(\bx)$ is the unique subgradient of $f$ at $\bx$, and conversely if $f$ has a unique subgradient at $\bx$, then $f$ is differentiable at that point (\citep{rockafellar1970convex}, Theorem 25.1).

The set-valued subdifferential mapping $\dom \partial f \ni \bx \mapsto \partial f(\bx)$ satisfies two important properties that will be used in this paper. First, it is monotone in that if $f$ is strongly convex of parameter $m \geqslant 0$ for every pair $(\bx_{0},\bx_{1}) \in \dom\partial f \times \dom\partial f$ and $\bp_{0}\in\partial f(\bx_{0})$, $\bp_{1}\in\partial f(\bx_{1})$, the following inequality holds (\cite{rockafellar1970convex}, page 240 and Corollary 31.5.2):
\begin{equation} \label{eq:monotone_mapping}
\left\langle \bp_{0}-\bp_{1},\bx_{0}-\bx_{1}\right\rangle \geqslant m\normsq{\bx_0-\bx_1}.
\end{equation}
Second, for every $\bx \in \dom\partial f$ the subdifferential $\partial f(\bx)$ is a closed convex set, and as a consequence, the mapping $\dom\partial f \ni \bx \mapsto \pi_{\partial f(\bx)}(\boldsymbol{0})$ selects the subgradient of the minimal norm in $\partial f(\bx)$ and defines a function continuous almost everywhere on $\dom\partial f$, which is a consequence of the fact that this mapping agrees with the gradient of $f$ over the set of points in $\inter{\dom J}$ at which $f$ is differentiable (\cite{rockafellar1970convex}, Theorem~25.5).
\end{defn}
\begin{defn}[Fenchel--Legendre transform] \label{def:legendre_t} 
Let $f\in\gmRn$. The \textit{Fenchel--Legendre transform} $f^{*}\colon\Rn\to\mathbb{R\cup}\{+\infty\}$ of $f$ is defined by
\begin{equation}
f^{*}(\bp)=\sup_{\bx\in\Rn}\left\{ \left\langle \bp,\bx\right\rangle -f(\bx)\right\} .\label{eq:fenchel_t_def}
\end{equation}
For every $f\in\gmRn$, the mapping $f\mapsto f^{*}$ is one-to-one, $f^{*}\in\gmRn$, and $(f^{*})^{*}=f$. Moreover, for every $\bx\in\Rn$ and $\bp\in\Rn$, $f$ and $f^*$ satisfy Fenchel's inequality
\begin{equation} \label{eq:fenchel_ineq}
f(\bx)+f(\bp)\geqslant\left\langle \bx,\bp\right\rangle,
\end{equation}
where equality holds if and only if $\bp\in\partial f(\bx)$, if and only if $\bx\in\partial f^{*}(\bp)$ (\citep{hiriart2013convexII}, corollary 1.4.4). If $f$ is also differentiable, the supremum in (\ref{eq:fenchel_t_def}) is attained whenever there exists $\bx\in\Rn$ such that $\bp=\nabla f(\bx)$. 

\end{defn}
\begin{defn}[Infimal convolutions] \label{def:inf_conv}
Let $f_{1}\in\gmRn$ and $f_{2}\in\gmRn$. The \textit{infimal convolution} of $f_{1}$ and $f_{2}$ is the function
\begin{equation}\label{eq:infconv_def}
\Rn \ni \bx \mapsto (f_{1} \Box f_{2})(\bx) = \inf_{\bx_{1}+\bx_{2}=\bx}\left\{ f_{1}(\bx_{1})+f_{2}(\bx_{2})\right\}.
\end{equation}
The infimal convolution is exact if the infimum is attained at $\bx_{1}\in\dom f_{1}$ and $\bx_{2}\in\dom f_{2}$, and in that case the infimum in \eqref{eq:infconv_def} can be replaced by a minimum. The Fenchel--Legendre transform of the infimal convolution (\ref{eq:infconv_def}) is the sum of their respective Fenchel--Legendre transforms (\citep{rockafellar1970convex}, Theorem 16.4), that is,
\[
\left(f_{1}\Box f_{2}\right)^{*}(\bp)=f_{1}^{*}(\bp)+f_{2}^{*}(\bp).
\]

If $f \in \gmRn$, then Moreau's Theorem \cite{moreau1965proximite,hiriart1989moreau} asserts that
\[
\frac{1}{2}\normsq{\cdot}\Box f + \frac{1}{2}\normsq{\cdot}\Box f^* = \frac{1}{2}\normsq{\cdot},
\]
and conversely, if $g$ and $h$ are two convex functions on $\Rn$ such that $g + h = \frac{1}{2}\normsq{\cdot}$, then there exists a unique function $F \in \gmRn$ such that
\[
g = \frac{1}{2}\normsq{\cdot}\Box F,\, \mbox{and }\, h = \frac{1}{2}\normsq{\cdot}\Box F^*,
\]
where $F(\bx) = h^*(\bx) - \frac{1}{2}\normsq{\bx}$ for every $\bx \in \Rn$, and
moreover $g$ and $f$ are continuously differentiable and
\[
\nabla g(\bx)\in\partial F(\nabla h(\bx))\quad\mbox{and}\quad\nabla h(\bx)\in\partial F^{*}(\nabla g(\bx)).
\]
\end{defn}
\begin{defn}[Bregman divergences] \label{def:breg_div} 
Let $f\in\gmRn$. The \textit{Bregman divergence} of $f$ is the function $D_{f}\colon\Rn\times\Rn\to\R\cup\{+\infty\}$ defined by
\begin{equation}
D_{f}(\bx,\bp)=f(\bx)-\left\langle \bp,\bx\right\rangle +f^{*}(\bp).\label{eq:breg_good_def}
\end{equation}
It satisfies $D_{f}(\bx,\bp)\geqslant0$ for every $\bx\in\Rn$ and $\bp\in\Rn$ by Fenchel's inequality (\ref{eq:fenchel_ineq}), with $D_{f}(\bx,\bp)=0$ whenever $\bp\in\partial f(\bx)$. It also satisfies $D_{f}(\bx,\bp)=D_{f^{*}}(\bp,\bx)$, with $D_{f}(\bx,\bp)=D_{f}(\bp,\bx)$ if and only if $f$ is the quadratic $f=\frac{1}{2}\left\Vert \cdot\right\Vert _{2}^{2}$.
\end{defn}
\begin{defn}[Moreau--Yosida envelopes and proximal mappings]\label{def:moreau_yosida}
Let $t > 0$, and $J \in \gmRn$. The functions
\begin{equation}
\bx\mapsto \inf_{\by\in\Rn}\left\{ \frac{1}{2t}\normsq{\bx-\by}+J(\by)\right\} 
\end{equation}
and
\begin{equation}
\bx\mapsto \argmin_{\by\in\Rn}\left\{ \frac{1}{2t}\normsq{\bx-\by}+J(\by)\right\}
\end{equation}
are called the Moreau--Yosida envelope and proximal mapping of $J$, respectively \cite{hiriart2013convexII,moreau1965proximite,rockafellar2009variational}. Their properties have been extensively studied in convex and functional analysis, and they form the basis for the mathematical analysis of the convex variational imaging model \eqref{eq:minimization_prob} and corresponding minimizer \eqref{eq:minimizer_prob}~\cite{darbon2015convex}. The following theorem describes the behavior of both the solution to the infimum problem \eqref{eq:minimization_prob} and its corresponding minimizer \eqref{eq:minimizer_prob}, and it shows in particular that for any observed image $\bx\in\Rn$ and parameter $t>0$, the imaging problem \eqref{eq:minimization_prob} has always a unique solution. The readers may refer to \citet{darbon2015convex} for more details.
\begin{thm} \label{thm:e_u_nonvisc_hj}
Suppose $J$ is a proper, lower semicontinuous and convex function, i.e., $J\in\gmRn$. Then the following statements hold.
\begin{enumerate}
\item[(i)]  The unique continuously differentiable and convex function $S_0\colon\Rn\times[0,+\infty)\to\R$ that satisfies the first-order Hamilton--Jacobi equation with initial data
\begin{equation}
\begin{dcases}
\frac{\partial S_0}{\partial t}(\bx,t)+\frac{1}{2}\normsq{\nabla_{\bx}S_0(\bx,t)}=0 & \mbox{in }\Rn\times(0,+\infty),\\
S_0(\bx,0)=J(\bx) & \mbox{in }\Rn,
\end{dcases}\label{eq:hj_non_visc_pde}
\end{equation}
is defined by
\begin{align}
S_0(\bx,t) & =\left(\frac{t}{2}\normsq{\cdot}+J^{*}\right)^{*}(\bx)\quad \ \ (\mbox{Hopf formula})\label{eq:hopf_formula}\\
& = \left((\frac{t}{2}\normsq{\cdot})^* \Box J\right)(\bx) \qquad(\mbox{Lax--Oleinik formula}) \\
 & =\inf_{\by\in\Rn}\left\{ \frac{1}{2t}\normsq{\bx-\by}+J(\by)\right\} .\label{eq:lax_formula}
\end{align}
Furthermore, for every $\bx\in\dom J$, sequence $\{t_k\}_{k=1}^{+\infty}$ of positive real numbers converging to $0$, and sequence $\{\boldsymbol{d}_k\}_{k=1}^{+\infty}$ of vectors converging to $\boldsymbol{d}\in\Rn$, the pointwise limit $S_0(\bx+t_k\boldsymbol{d}_k,t_k)$ as $k \to +\infty$ exists and satisfies
\[
\lim_{k \to +\infty} S_0(\bx+t_k\boldsymbol{d}_k,t_k)=J(\bx).
\]
\item[(ii)]  For every $\bx\in\Rn$ and $t>0$, the infimum in (\ref{eq:lax_formula}) exists and is attained at a unique point $\bu_{MAP}(\bx,t)\in\dom\partial J$. In addition, the minimizer $\bu_{MAP}(\bx,t)$ satisfies the formula
\begin{equation} \label{eq:grad_map_rep}
\bu_{MAP}(\bx,t)=\bx-t\nabla_{\bx}S_0(\bx,t),
\end{equation}
and $\left(\frac{\bx-\bu_{MAP}(\bx,t)}{t}\right)\in\partial J(\bu_{MAP}(\bx,t))$.
\item[(iii)] For every $\bx \in \dom J$, the pointwise limit of $\bu_{MAP}(\bx,t)$ as $t\to0$ exists and satisfies
\[
\lim_{\substack{t\to0\\t>0}} \bu_{MAP}(\bx,t)=\bx.
\]

\item[(iv)] Let $\bx \in \dom\partial J$ and let $\{t_k\}_{k=1}^{+\infty}$ be a sequence of positive real numbers converging to zero. Then the limit of $\nabla_{\bx}S_0(\bx,t_k)$ as $k\to +\infty$ exists and satisfies
\begin{equation} \label{eq:grad_limit_subdiff}
    \lim_{k\to +\infty}\nabla_{\bx}S_0(\bx,t_k) = \pi_{\partial J(\by)}(\boldsymbol{0}).
\end{equation}
\end{enumerate}
\end{thm}

\begin{proof}
The proof of (i) relies on convex analysis; see \cite{hiriart2012optimisation} for a detailed proof. A proof of (ii)-(iv) can be found in \cite{darbon2015convex} (Lemma 2.1, Proposition 3.1, Proposition 3.2 (i), and Proposition 3.4).
\end{proof}
\end{defn}

\section{Connections between Bayesian posterior mean estimators and Hamilton--Jacobi partial differential equations}  \label{sec:connections}

\subsection{Set-up} \label{subsec:set-up}
To establish connections between Bayesian posterior mean estimators and Hamilton--Jacobi equations, we will assume that the regularization term $J$ in the variational imaging model~\eqref{eq:minimization_prob} satisfies the following assumptions:
\begin{enumerate}
\item[(A1)]  $J\in\gmRn$,
\item[(A2)]  $\inter(\dom J)\neq \varnothing$,
\item[(A3)]  $\inf_{\by\in\Rn}J(\by)<+\infty$,
and without loss of generality, $\inf_{\by\in\Rn}J(\by)=0$.
\end{enumerate}
Assumption (A1) ensures that the minimal value of the convex imaging problem \eqref{eq:minimization_prob} and its minimizer \eqref{eq:minimizer_prob} are well-defined and enjoy several properties (see Section~\ref{sec:background}, Definition 10, Theorem~\ref{thm:e_u_nonvisc_hj}). Assumption (A2) ensures that for every $\bx \in \Rn$, $t>0$, and $\epsilon > 0$, the posterior distribution
\begin{equation} \label{eq:prob_dist_j}
    \Rn\times\Rn\times(0,+\infty)\times(0,+\infty) \ni (\by,\bx,t,\epsilon) \mapsto q(\by|\bx,t,\epsilon) = \frac{e^{-\left(\frac{1}{2t}\left\Vert \bx-\by\right\Vert _{2}^{2}+J(\by)\right)/\epsilon}}{\int_{\Rn}e^{-\left(\frac{1}{2t}\left\Vert \bx-\by\right\Vert _{2}^{2}+J(\by)\right)/\epsilon}\diff\by}
\end{equation}
and its associated partition function
\begin{equation} \label{eq:partition_function}
\Rn\times(0,+\infty)\times(0,+\infty)\ni(\bx,t,\epsilon)\mapsto Z_{J}(\bx,t,\epsilon)=\int_{\Rn}e^{-\left(\frac{1}{2t}\left\Vert \bx-\by\right\Vert _{2}^{2}+J(\by)\right)/\epsilon}\diff\by
\end{equation}
are well-defined, and finally assumption (A3) guarantees that the partition function~\eqref{eq:partition_function} is also bounded from above independently of $\bx \in \Rn$. Additional requirements beyond that $J\in\gmRn$ are necessary because the integral in \eqref{eq:partition_function} adds a measure-theoretic aspect largely absent from the convex minimization problem \eqref{eq:minimization_prob}. For convenience, in the rest of this paper we will denote the expected value of a measurable function $f\colon(\Rn,\mathcal{B}(\Rn))\to(\R^{m},\mathcal{B}(\R^{m}))$ (with $m\in\mathbb{N}$) by
\begin{equation} \label{eq:expectation}
\expectationJ{f(\by)}=\frac{1}{Z_{J}(\bx,t,\epsilon)}\int_{\Rn}f(\by)e^{-\left(\frac{1}{2t}\left\Vert \bx-\by\right\Vert _{2}^{2}+J(\by)\right)/\epsilon}\diff\by.
\end{equation}
Thus, we write $\expectationJ{\by}$ for the posterior mean estimate and $\expectationJ{\normsq{\by - \bu_{PM}(\bx,t,\epsilon)}}$ for the MSE of the posterior distribution~\eqref{eq:prob_dist_j}, respectively.

\subsection{Connections to second-order Hamilton--Jacobi equations} \label{subsec:HJ_eqns_visc}
The next theorem establishes connections between viscous HJ PDEs with initial data $J$ satisfying assumptions (A1)-(A3) and both the partition function~\eqref{eq:partition_function} and the Bayesian posterior mean estimate~\eqref{eq:pm_estimate}. These connections mirror those between the first-order HJ PDE~\eqref{eq:hj_non_visc_pde} with initial data $J$ satisfying assumption (A1) and both the convex minimization problem~\eqref{eq:minimization_prob} and the MAP estimate~\eqref{eq:minimizer_prob}. The connections between viscous HJ PDEs and Bayesian posterior mean estimators will be leveraged later to describe various properties of posterior mean estimators in terms of the observed image $\bx$ and parameters $t$ and $\epsilon$, and in particular in Section~\ref{subsec:HJ_eqns_1order} to show that the posterior mean estimate~\eqref{eq:pm_estimate} can be expressed as the minimizer associated to the solution to a first-order HJ PDE (Theorem~\ref{thm:char_opt_sol}) with twice continuously differentiable and convex regularization term.

\begin{thm}[The viscous Hamilton--Jacobi equation with initial data in $\gmRn$] \label{thm:visc_hj_eqns} Suppose the function $J$ satisfies assumptions (A1)-(A3). Then the following statements hold.
\begin{enumerate}
\item[(i)]  For every $\epsilon>0$, the function $S_\epsilon\colon\Rn\times[0,+\infty)\to[0,+\infty)$ defined by
\begin{equation} \label{eq:hj_s_eps}
S_{\epsilon}(\bx,t)\coloneqq-\epsilon\ln\left(\frac{1}{\left(2\pi t\epsilon\right)^{n/2}}Z_J(\bx,t,\epsilon)\right)=-\epsilon\ln\left(\frac{1}{\left(2\pi t\epsilon\right)^{n/2}}\int_{\Rn}e^{-\left(\frac{1}{2t}\left\Vert \bx-\by\right\Vert _{2}^{2}+J(\by)\right)/\epsilon}\diff\by\right)
\end{equation}
is the unique smooth solution to the second-order Hamilton Jacobi PDE with initial data
\begin{equation}
\begin{dcases}\label{eq:hj_visc_pde}
\frac{\partial S_{\epsilon}}{\partial t}(\bx,t)+\frac{1}{2}\left\Vert \nabla_{\bx}S_{\epsilon}(\bx,t)\right\Vert _{2}^{2}=\frac{\epsilon}{2}\Laplacian S_{\epsilon}(\bx,t) & \text{ in }\Rn\times(0,+\infty),\\
S_{\epsilon}(\bx,0)=J(\bx) & \text{ in }\Rn.
\end{dcases}
\end{equation}
In addition, the domain of integration in \eqref{thm:visc_hj_eqns} can be taken to be $\dom J$ or, up to a set of Lebesgue measure zero, $\inter(\dom J)$ or $\dom(\partial J)$. Furthermore, for every $\bx\in\dom J$ and $\epsilon > 0$, except possibly at the boundary points $\bx\in(\dom J)\backslash(\inter(\dom J))$ if such points exist, the pointwise limit $S_\epsilon(\bx,t)$ as $t \to 0$ exists and satisfies 
\begin{equation}
    \lim_{\substack{t\to0\\t>0}}S_{\epsilon}(\bx,t)=e^{-J(\bx)/\epsilon}.
\end{equation}
If $\bx\in(\dom J)\backslash(\inter(\dom J))$, then we may only conclude 
\[
\liminf_{\substack{t\to0\\t>0}}S_{\epsilon}(\bx,t)\geqslant J(\bx)
\]
 and 
\[\limsup_{\substack{t\to0\\t>0}}S_{\epsilon}(\bx,t)\leqslant J(\bx)-\epsilon\ln\left(\frac{1}{(2\pi\epsilon)^{n/2}}\int_{\dom J}e^{-\frac{1}{2\epsilon}\left\Vert \bx-\by\right\Vert _{2}^{2}}\diff\by\right).
\]

\item[(ii)]  (Convexity and monotonicity properties).
\begin{enumerate}
    \item The function $\Rn\times(0,+\infty)\ni(\bx,t)\mapsto S_{\epsilon}(\bx,t)-\frac{n\epsilon}{2}\ln t$ is jointly convex.
    \item The function $(0,+\infty)\ni t\mapsto S_{\epsilon}(\bx,t)-\frac{n\epsilon}{2}\ln t$ is strictly monotone decreasing.
    \item The function $(0,+\infty)\ni\epsilon\mapsto S_{\epsilon}(\bx,t)-\frac{n\epsilon}{2}\ln\epsilon$ is strictly monotone decreasing.
    \item The function $\Rn \ni \bx \mapsto \frac{1}{2}\normsq{\bx} - tS_\epsilon(\bx,t)$ is strictly convex.
\end{enumerate}

\item[(iii)] (Connections to the posterior mean and MSE) The posterior mean estimate $\bu_{PM}(\bx,t,\epsilon)$ and the MSE $\expectationJ{\left \Vert \by - \bu_{PM}(\bx,t,\epsilon) \right \Vert_{2}^{2}}$ satisfy the formulas
\begin{equation} \label{eq:grad_s_eps}
    \bu_{PM}(\bx,t,\epsilon) = \bx - t\nabla_{\bx}S_\epsilon(\bx,t)
\end{equation}
and
\begin{equation}\label{eq:exact_variance}
\begin{alignedat}{1}
    \expectationJ{\left \Vert \by - \bu_{PM}(\bx,t,\epsilon) \right \Vert_{2}^{2}} &= t\epsilon \nabla_{\bx}\cdot \bu_{PM}(\bx,t,\epsilon) \\
    &= nt\epsilon - t^2\epsilon\Laplacian S_{\epsilon}(\bx,t).
\end{alignedat}
\end{equation}
Moreover, $\bx \mapsto \bu_{PM}(\bx,t,\epsilon)$ is a bijective function.

\item[(iv)] (Vanishing $\epsilon \to 0$ limit) Let $S_0:\Rn \times (0,+\infty) \to \R$ denote the continuously differentiable and convex solution to the first-order HJ PDE~\eqref{eq:hj_non_visc_pde} with initial data $J$. For every $\bx\in\Rn$ and $t>0$, the following limit holds:
\begin{equation} \label{eq:ldp_limit}
\lim_{\substack{\epsilon\to0\\\epsilon>0}} -\epsilon\ln\left(\frac{1}{\left(2\pi t\epsilon\right)^{n/2}}\int_{\Rn}e^{-\left(\frac{1}{2t}\left\Vert \bx-\by\right\Vert _{2}^{2}+J(\by)\right)/\epsilon}\diff\by\right) = \inf_{\by\in\Rn}\left\{ \frac{1}{2t}\left\Vert \bx-\by\right\Vert _{2}^{2}+J(\by)\right\},
\end{equation}
that is,
\[
\lim_{\substack{\epsilon\to0\\\epsilon>0}}S_{\epsilon}(\bx,t) = S_0(\bx,t),
\]
and the limit converges uniformly over every compact set of $\Rn\times(0,+\infty)$ in $(\bx,t)$. In addition, the gradient $\nabla_{\bx}S_{\epsilon}(\bx,t)$, the partial derivative $\frac{\partial S_{\epsilon}(\bx,t)}{\partial t}$, and the Laplacian $\frac{\epsilon}{2}\Laplacian S_{\epsilon}(\bx,t)$ satisfy the limits
\[
\lim_{\substack{\epsilon\to0\\\epsilon>0}}\nabla_{\bx}S_{\epsilon}(\bx,t)=\nabla_{\bx}S_0(\bx,t), \quad\lim_{\substack{\epsilon\to0\\
\epsilon>0}}\frac{\partial S_{\epsilon}}{\partial t}(\bx,t)=\frac{\partial S_0}{\partial t}(\bx,t),
\]
and
\[
\lim_{\substack{\epsilon\to0\\
\epsilon>0}}\frac{\epsilon}{2}\Laplacian S_{\epsilon}(\bx,t)=0,
\]
where each limit converges uniformly over every compact set of $\Rn\times(0,+\infty)$ in $(\bx,t)$. As a consequence, for every $\bx\in\Rn$ and $t>0$, the pointwise limit of $\bu_{PM}(\bx,t,\epsilon)$ as $\epsilon\to 0$ exists and satisfy
\[
\lim_{\substack{\epsilon\to0\\\epsilon>0}} \bu_{PM}(\bx,t,\epsilon) = \bu_{MAP}(\bx,t),
\]
and the limit converges uniformly over every compact set of $\Rn\times(0,+\infty)$ in $(\bx,t)$.
\end{enumerate}
\end{thm}

\begin{proof}
See Appendix~\ref{app:A}.
\end{proof}

\begin{rem}
Note that solutions to \eqref{subsec:HJ_eqns_visc} exist under weaker conditions by weakening assumptions (A1) and (A3), but then global existence, pointwise limit to the initial condition (almost everywhere), boundedness, and log--concavity properties of solutions may no longer hold.
\end{rem}

To illustrate certain aspects of Theorem \ref{thm:visc_hj_eqns} and properties of posterior mean estimates, we give here two analytical examples.

\begin{example}[Tikhonov--Phillips regularization] \label{example:tikhonov}
Let $J(\bx) = \frac{m}{2}\left \Vert \bx \right \Vert_{2}^{2}$ with $m>0$, and consider the solution $S_0(\bx,t)$ and $S_{\epsilon}(\bx,t)$ to the first-order PDE \eqref{eq:hj_non_visc_pde} and viscous HJ PDE \eqref{eq:hj_visc_pde} with initial data $J$, respectively. 

The solution $S_0(\bx,t)$ is given by the Lax--Oleinik formula
\begin{alignat*}{1}
S_0(\bx,t) & =\inf_{\by\in\Rn}\left\{ \frac{1}{2t}\left\Vert \bx-\by\right\Vert _{2}^{2}+\frac{m}{2}\left\Vert \by\right\Vert _{2}^{2}\right\} \\
& =\frac{m\left\Vert \bx\right\Vert _{2}^{2}}{2(1+mt)}.
\end{alignat*}
This minimization problem is a special case of Tikhonov--Phillips regularization (also known as ridge regression in statistics), a method for regularizing ill-posed problems in inverse problems and statistics using a quadratic regularization term \cite{phillips1962technique,tikhonov1995numerical}. The corresponding minimizer can be computed using the gradient $\nabla_{\bx}S_0(\bx,t)$ via equation \eqref{eq:grad_map_rep} in Theorem \ref{thm:visc_hj_eqns}:
\begin{equation*}
\bu_{MAP}(\bx,t)  = \bx - t\nabla_{\bx}S_0(\bx,t) = \bx - \frac{mt\bx}{1+mt} = \frac{\bx}{1+mt}.
\end{equation*}

The solution $S_\epsilon(\bx,t)$ is given by the integral
\begin{alignat*}{1}
S_{\epsilon}(\bx,t) & =-\epsilon\ln\left(\frac{1}{(2\pi t\epsilon)^{n/2}}\int_{\Rn}e^{-\left(\frac{1}{2t}\left\Vert \bx-\by\right\Vert _{2}^{2}+\frac{m}{2}\left\Vert \by\right\Vert _{2}^{2}\right)}\diff\by\right)\\
 & =\frac{m\left\Vert \bx\right\Vert ^{2}}{2(1+mt)}+\frac{n\epsilon}{2}\ln\left(1+mt\right).
\end{alignat*}
The posterior mean estimate $\bu_{PM}(\bx,t,\epsilon)$ can be computed using the representation formula~\eqref{eq:grad_s_eps} in Theorem~\ref{thm:visc_hj_eqns}(iii) by calculating the gradient $\nabla_{\bx}S_\epsilon(\bx,t)$:
\begin{equation*}
\bu_{PM}(\bx,t,\epsilon)  = \bx - t\nabla_{\bx}S_\epsilon(\bx,t) = \bx - \frac{mt\bx}{1+mt} = \frac{\bx}{1+mt}.
\end{equation*}
The MSE $\expectationJ{\normsq{\by-\bu_{PM}(\bx,t,\epsilon)}}$ can be computed using the representation formula~\eqref{eq:exact_variance} in Theorem~\ref{thm:visc_hj_eqns}(iii) by calculating the divergence of $\bu_{PM}(\bx,t,\epsilon)$:
\begin{equation}\label{eq:tikhonov_var}
    \expectationJ{\left \Vert \by - \bu_{PM}(\bx,t,\epsilon)\right \Vert_{2}^{2}} = t\epsilon \nabla_{\bx} \cdot \bu_{PM}(\bx,t,\epsilon) = \frac{nt\epsilon}{1+mt}.
\end{equation}

Comparing the solutions $S_0(\bx,t)$ and $S_\epsilon(\bx,t)$, we see that $\lim_{\substack{\epsilon\to0\\
\epsilon>0}}S_{\epsilon}(\bx,t)=S_0(\bx,t)$ for every $\bx\in\Rn$ and $t>0$, in accordance to the result established in Theorem~\ref{thm:visc_hj_eqns}(iv). Note also that while $(\bx,t)\mapsto S_0(\bx,t)$ is jointly convex, its viscous counterpart $(\bx,t)\mapsto S_\epsilon(\bx,t)$ is not. Indeed, $t \mapsto S_\epsilon(\bx,t)$ is not convex. It is convex only after subtracting $\frac{n\epsilon}{2}\ln t$ from $S_{\epsilon}(\bx,t)$. This implies that the joint convexity result Theorem \ref{thm:visc_hj_eqns}(ii)(a) is sharp.
\end{example}

\begin{example}[Soft thresholding] \label{ex:smooth_shrink} 
Let $J(\bx)=\sum_{i=1}^{n}\lambda_{i}\left|\bx_{i}\right|$, where $\lambda_{i}>0$ for each $i\in\{1,\dots,n\}$, and consider the solutions $S_0(\bx,t)$ and $S_{\epsilon}(\bx,t)$ to the first-order PDE \eqref{eq:hj_non_visc_pde} and second-order PDE \eqref{eq:hj_visc_pde} with initial data $J$, respectively.

The solution $S_0(\bx,t)$ is given by the Lax--Oleinik formula
\begin{align*}
S_{0}(\bx,t) & =\inf_{\by\in\Rn}\left\{ \frac{1}{2t}\left\Vert \bx-\by\right\Vert _{2}^{2}+\sum_{i=1}^{n}\lambda_{i}\left|y_{i}\right|\right\} \label{eq:ex_sol_eps_0}\\
 & =\sum_{i=1}^{n}\left(\inf_{y_{i}\in\R}\left\{ \frac{1}{2t}(x_{i}-y_{i})^{2}+\lambda_{i}\left|y_{i}\right|\right\} \right),
\end{align*}
where $x_i$ and $y_i$ denote the $i^{\rm th}$ component of the vectors $\bx$ and $\by$, respectively. In the context of imaging, this minimization problem corresponds to denoising an image with the weighted sum of a quadratic fidelity term and a weighted $l_{1}$-norm as the regularization term. This term is widely used in imaging to encourage sparsity of an image, and it has received considerable interest due to its connection with compressed sensing reconstruction \citep{candes2006robust,donoho2006compressed}. The solution to this minimization problem corresponds to a soft thresholding applied component-wise to the vector $\bx$ \cite{daubechies2004iterative, figueiredo2001wavelet,lions1979splitting}. The soft thresholding operator is defined for any real number $a$ and positive real number $\alpha$ as
\begin{equation} \label{eq:thresholding_operator}
\R \times (0,+\infty) \ni (a,\alpha) \mapsto T(a,\alpha)=\begin{cases}
a-\alpha & \mbox{if }a>\alpha,\\
0 & \mbox{if }a\in[-\alpha,\alpha],\\
a+\alpha & \mbox{if }a<-\alpha.
\end{cases}
\end{equation}
The minimizer in the Lax--Oleinik formula of $S_0(\bx,t)$ is then given component-wise by
\[
\left(\bu_{MAP}(\bx,t)\right)_{i}=T(x_{i},t\lambda_{i}),
\]
so that
\begin{equation*}
S_{0}(\bx,t) = \sum_{i=1}^{n}\left( \frac{1}{2t}(x_{i}-T(x_{i},t\lambda_{i}))^{2}+\lambda_{i}\left|T(x_{i},t\lambda_{i})\right|\right).
\end{equation*}

The solution $S_\epsilon(\bx,t)$ is given by the integral
\begin{align*}
S_{\epsilon}(\bx,t) & =-\epsilon\ln\left(\frac{1}{(2\pi t\epsilon)^{n/2}}\int_{\Rn}e^{-\left(\frac{1}{2t}\left\Vert \bx-\by\right\Vert _{2}^{2}+\sum_{k=1}^{n}\lambda_{i}\left|y_{i}\right|\right)/\epsilon}\diff\by\right)\\
 & =-\epsilon\sum_{i=1}^{n}\ln\left(\frac{1}{2}\sqrt{\frac{2}{\pi t\epsilon}}\int_{-\infty}^{+\infty}e^{-\left(\frac{1}{2t}(x_{i}-y_i)^{2}+\lambda_{i}\left|y_i\right|\right)/\epsilon}\diff y_i\right)\\
 & =-\epsilon\sum_{i=1}^{n}\ln\left(\frac{1}{2}\sqrt{\frac{2}{\pi t\epsilon}}\left(\int_{0}^{+\infty}e^{-\left(\frac{1}{2t}(x_{i}+y_i)^{2}+\lambda_{i}y_i\right)/\epsilon}\diff y_i + \int_{0}^{+\infty}e^{-\left(\frac{1}{2t}(x_{i}-y_i)^{2}+\lambda_{i}y_i\right)/\epsilon}\diff y_i\right)\right)
\end{align*}
To compute this integral, first define the function
\begin{equation*}
    \R \ni z \mapsto L(z) = \frac{1}{2}e^{z^2}\erfc\left(z\right),
\end{equation*}
where $\erfc$ denotes the complementary error function. Then we have (\citep{gradshte2007table}, page 336, integral 3.332, 2., and page 887, integral 8.250, 1.)
\begin{equation*}
    \frac{1}{2}\sqrt{\frac{2}{\pi t\epsilon}}\int_{0}^{+\infty}e^{-\left(\frac{1}{2t}(x_{i}+y_i)^{2}+\lambda_{i}y_i\right)/\epsilon}\diff y_i = e^{-\frac{x_{i}^2}{2t\epsilon}}L\left(\frac{x_i + t\lambda_i}{\sqrt{2t\epsilon}}\right)
\end{equation*}
and
\begin{equation*}
    \frac{1}{2}\sqrt{\frac{2}{\pi t\epsilon}}\int_{0}^{+\infty}e^{-\left(\frac{1}{2t}(x_{i}-y_i)^{2}+\lambda_{i}y_i\right)/\epsilon}\diff y_i = e^{-\frac{x_{i}^2}{2t\epsilon}}L\left(\frac{-x_i + t\lambda_i}{\sqrt{2t\epsilon}}\right),
\end{equation*}
from which we get
\begin{equation*}
S_{\epsilon}(\bx,t)=\frac{\left\Vert \bx\right\Vert _{2}^{2}}{2t} -\epsilon\sum_{i=1}^{n}\ln\left(L\left(\frac{x_i + t\lambda_i}{\sqrt{2t\epsilon}}\right) + L\left(\frac{-x_i + t\lambda_i}{\sqrt{2t\epsilon}}\right)\right).
\end{equation*}
Now, to compute the posterior mean estimate it suffices to compute the gradient of $\nabla_{\bx}S_\epsilon(\bx,t)$ and use the formula $\bu_{PM}(\bx,t,\epsilon) = \bx - t\nabla_{\bx}S_\epsilon(\bx,t)$. To do so, we must compute the derivative of the function $L$. Since 
\begin{equation*}
    \frac{dL}{dz}(z) = 2zL(z)+\frac{1}{\sqrt{\pi}},
\end{equation*}
the chain rule gives 
\begin{equation*}
    \frac{\partial}{\partial x_i}\left(L\left(\frac{x_i + t\lambda_i}{\sqrt{2t\epsilon}}\right) + L\left(\frac{-x_i + t\lambda_i}{\sqrt{2t\epsilon}}\right)\right)  = \left(\frac{x_i + t\lambda_i}{t\epsilon}\right)L\left(\frac{x_i + t\lambda_i}{\sqrt{2t\epsilon}}\right) - \left(\frac{-x_i + t\lambda_i}{t\epsilon}\right)L\left(\frac{-x_i + t\lambda_i}{\sqrt{2t\epsilon}}\right).
\end{equation*}
The posterior mean estimate is therefore given component-wise by
\begin{alignat*}{1}
    (\bu_{PM}(\bx,t,\epsilon))_{i} & = x_i - t(\nabla_{\bx}S_\epsilon(\bx,t))_i\\
    & = x_i + t\lambda_i\left(\frac{L\left(\frac{x_i + t\lambda_i}{\sqrt{2t\epsilon}}\right) + L\left(\frac{-x_i + t\lambda_i}{\sqrt{2t\epsilon}}\right)}{L\left(\frac{x_i + t\lambda_i}{\sqrt{2t\epsilon}}\right) - L\left(\frac{-x_i + t\lambda_i}{\sqrt{2t\epsilon}}\right)}\right)
\end{alignat*}

The posterior mean estimate $\bu_{PM}(\bx,t,\epsilon)$ yields a smooth analogue of the soft thresholding operator $T$ (defined in \eqref{eq:thresholding_operator}) evaluated at $(x_i,t\lambda_i)$, in the sense that $\lim_{\substack{\epsilon\to0\\\epsilon>0}} (\bu_{PM}(\bx,t,\epsilon))_i = T(x_{i},t\lambda_{i})$ for every $i\in\{1,\dots,n\}$ by Theorem \ref{thm:visc_hj_eqns}(iv). Figure~\ref{fig:sparsity_example} shows the MAP and posterior mean estimates in one dimension for the choice of $t=1.25$, $\epsilon=\left\{ 0.025,\,0.1,\,0.25,\,0.5,\,1\right\} $, and $\lambda_{1}=2$ for $x\in[-5,5]$.
\begin{figure}[h]
\includegraphics[width=0.75\textwidth]{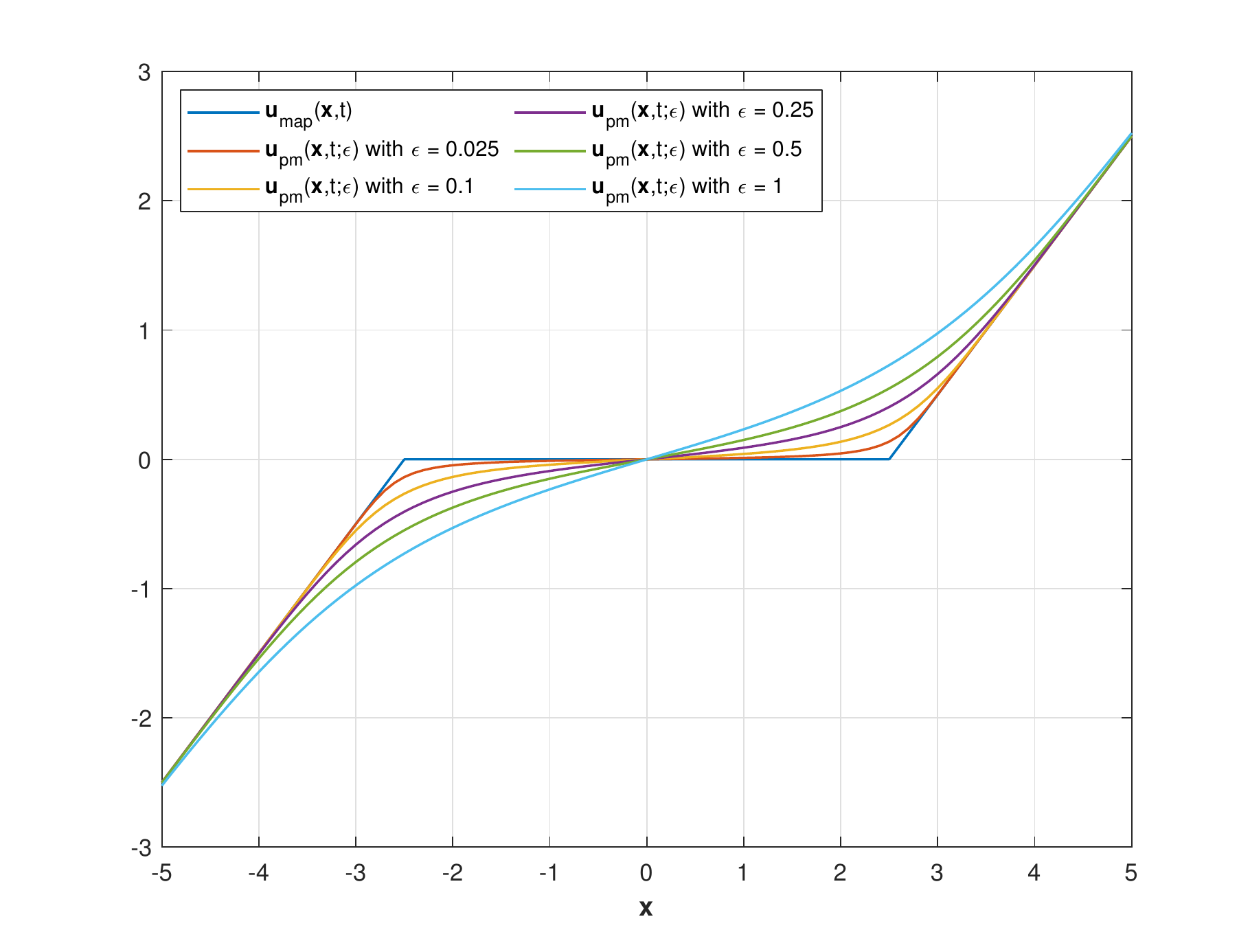}
\caption{Numerical example of the MAP and posterior mean estimates in one dimension with $J(x)=\lambda_{1}\left|x\right|$ for the choice of $t=1.25$, $\epsilon=\left\{ 0.025,\,0.1,\,0.25,\,0.5,\,1\right\}$, and $\lambda_{1}=2$ for $x\in[-5,5]$.}
\label{fig:sparsity_example}
\end{figure}
\end{example}

\subsection{Connections to first-order Hamilton--Jacobi equations} \label{subsec:HJ_eqns_1order}

In this section, we use the connections between the posterior mean estimate~\eqref{eq:pm_estimate} and viscous HJ PDEs established in Theorem~\ref{thm:visc_hj_eqns} to show that the posterior mean estimate can be expressed through the solution to a first-order HJ PDE with initial data of the form of~\eqref{eq:hj_non_visc_pde}. In particular, we show that the posterior mean estimate satisfies the proximal mapping formula
\begin{equation*}
\bu_{PM}(\bx,t,\epsilon) = \argmin_{\by\in\Rn}\left\{ \frac{1}{2}\left\Vert \bx-\by\right\Vert _{2}^{2} + \left(K^{*}_{\epsilon}(\by,t)-\frac{1}{2}\Vert \by \Vert_{2}^{2}\right)\right\},
\end{equation*}
where the function $K_\epsilon\colon\Rn\times\times(0,+\infty) \to \R$ is defined through the solution $S_\epsilon(\bx,t)$ to the viscous HJ PDE~\eqref{eq:hj_visc_pde} via
\begin{equation*}
K_{\epsilon}(\bx,t) \coloneqq \frac{1}{2}\left\Vert \bx\right\Vert _{2}^{2}-tS_{\epsilon}(\bx,t) \equiv t\epsilon\ln\left(\frac{1}{(2\pi t\epsilon)^{n/2}}\int_{\dom J}e^{\left(\frac{1}{t}\left\langle \bx,\by\right\rangle -\frac{1}{2t}\left\Vert \by\right\Vert _{2}^{2}-J(\by)\right)/\epsilon}\diff\by\right),
\end{equation*}
which is convex by Theorem~\ref{thm:visc_hj_eqns}(ii)(d), and where $K^{*}_{\epsilon}(\by,t)$ denotes the Fenchel--Legendre transform of $\by \mapsto K_\epsilon(\by,t)$. This result gives the representation of the convex imaging regularization term whose existence was derived by \cite{louchet2008modeles,gribonval2011should,gribonval2013reconciling,louchet2013posterior} (and later extended to non-Gaussian data fidelity terms in \cite{gribonval2018characterization,gribonval2018bayesian}). This representation result depends crucially on the connections established between the posterior mean estimate $\bu_{PM}(\bx,t,\epsilon)$ and the viscous HJ PDE~\eqref{eq:hj_visc_pde} established in Theorem~\ref{thm:visc_hj_eqns}. Moreover, we show that $\by \mapsto K^{*}_{\epsilon}(\by,t)$ is twice continuously differentiable. This fact has the important consequence that the posterior mean estimate $\bu_{PM}(\bx,t,\epsilon)$ for image denoising avoids staircasing effects thanks to a result established by \citet{nikolova2004weakly} (Theorem 3 in her paper, specifically). This result was proven for Total Variation regularization terms by \citet{louchet2008modeles} in a different manner; here our results are applicable to any regularization term $J$ satisfying assumptions (A1)-(A3).

\begin{thm}[Connections between the posterior mean estimate and first-order HJ PDEs] \label{thm:char_opt_sol}
Suppose the function $J$ satisfies assumptions (A1)-(A3). For every $\bx \in \Rn$, $t>0$, and $\epsilon>0$, let $S_\epsilon(\bx,t)$ denote the solution to the second-order HJ PDE \eqref{eq:hj_visc_pde} with initial data $J$ and let $\bu_{PM}(\bx,t,\epsilon)$ denote the posterior mean estimate \eqref{eq:pm_estimate}. Consider the first-order HJ PDE
\begin{equation} \label{eq:hj_pde_mod}
\begin{dcases}
\frac{\partial\tilde{S}}{\partial s}(\bx,s)+\frac{1}{2}\left\Vert \nabla_{\bx}\tilde{S}(\bx,s)\right\Vert _{2}^{2}=0 & \textrm{\text{ in }}\,\Rn\times(0,+\infty),\\
\tilde{S}(\bx,0)=K^{*}_{\epsilon}(\bx,t)-\frac{1}{2}\Vert \bx \Vert_{2}^{2} & \textrm{\text{ in }}\,\Rn.
\end{dcases}
\end{equation}
Then the initial data $\bx \mapsto K^{*}_{\epsilon}(\bx,t)-\frac{1}{2}\Vert \bx \Vert_{2}^{2}$ is convex, the solution to the HJ PDE~\eqref{eq:hj_pde_mod} satisfies the Lax--Oleinik formula
\begin{equation*}
\tilde{S}_0(\bx,s) = \inf_{\by\in\Rn}\left\{ \frac{1}{2s}\left\Vert \bx-\by\right\Vert _{2}^{2} + \left(K^{*}_{\epsilon}(\by,t)-\frac{1}{2}\Vert \by \Vert_{2}^{2}\right)\right\},
\end{equation*}
and the corresponding minimizer at $s = 1$ is the posterior mean estimate $\bu_{PM}(\bx,t,\epsilon)$:
\begin{equation}\label{eq:upm_rep_map_est}
\bu_{PM}(\bx,t,\epsilon) = \argmin_{\by\in\Rn}\left\{ \frac{1}{2}\left\Vert \bx-\by\right\Vert _{2}^{2} + \left(K^{*}_{\epsilon}(\by,t)-\frac{1}{2}\Vert \by \Vert_{2}^{2}\right)\right\}.
\end{equation}
\end{thm}
Moreover, for every $t>0$ and $\epsilon > 0$ the function $\Rn \ni \by \mapsto K^{*}_{\epsilon}(\by,t)$ is twice continuously differentiable.

\begin{proof}
By definition of the function $(\bx,t) \mapsto K_{\epsilon}(\bx,t)$, we may write
\begin{equation*}
    tS_\epsilon(\bx,t) + K_\epsilon(\bx,t) = \frac{1}{2}\Vert \bx \Vert_{2}^{2}.
\end{equation*}
As both $\bx \mapsto tS_\epsilon(\bx,t)$ and $\bx \mapsto K_\epsilon(\bx,t)$ are convex by Theorem~\ref{thm:visc_hj_eqns}(ii)(a) and (d), we can apply Moreau's decomposition theorem (see definition \ref{def:inf_conv} in Section \ref{sec:background}) to conclude that $\bx \mapsto K^{*}_{\epsilon}(\bx,t)-\frac{1}{2}\Vert \bx \Vert_{2}^{2}$ is convex and to express $tS_\epsilon(\bx,t)$ as
\begin{equation} \label{eq:s_eps_equiv}
    tS_\epsilon(\bx,t) = \inf_{\by \in \Rn}\left\{\frac{1}{2}\left \Vert \bx - \by \right \Vert_{2}^{2} + \left(K^{*}_{\epsilon}(\by,t) - \frac{1}{2}\left \Vert \by \right \Vert_{2}^{2}\right)\right\}
\end{equation}
On the one hand, by Theorem~\ref{thm:e_u_nonvisc_hj} the right hand side of \eqref{eq:s_eps_equiv} is the solution $\tilde{S}_0(\bx,s)$ to the first-order HJ PDE \eqref{eq:hj_pde_mod} at $s = 1$, and therefore its minimizer is given by $\bx - \nabla_{\bx}\tilde{S}_0(\bx,1)$. On the other hand, the gradient $\nabla_{\bx}\tilde{S}_0(\bx,1)$ is equal to the left hand side of \eqref{eq:s_eps_equiv}, that is, $\nabla_{\bx}\tilde{S}_0(\bx,1) = t\nabla_{\bx}S_\epsilon(\bx,t)$, which is equal to $\bx - \bu_{PM}(\bx,t,\epsilon)$ by formula~\eqref{eq:pm_estimate}. As a result, the posterior mean estimate $\bu_{PM}(\bx,t,\epsilon)$ minimizes the right hand side of \eqref{eq:s_eps_equiv}, that is, 
\begin{equation*}
\bu_{PM}(\bx,t,\epsilon) = \argmin_{\by\in\Rn}\left\{ \frac{1}{2}\left\Vert \bx-\by\right\Vert _{2}^{2} + \left(K^{*}_{\epsilon}(\by,t)-\frac{1}{2}\Vert \by \Vert_{2}^{2}\right)\right\}.
\end{equation*}

Now, using the strict convexity of $\bx \mapsto K_{\epsilon}(\bx,t)$ and that $\nabla K_\epsilon(\bx,t) = \bu_{PM}(\bx,t,\epsilon)$ is a bijective function in $\bx$ for every $t>0$ and $\epsilon>0$ by Theorem~\ref{thm:visc_hj_eqns} we can invoke (Theorem 26.5, \cite{rockafellar1970convex}) to conclude that $\by \mapsto K^{*}_{\epsilon}(\by,t)$ is a continuously differentiable, strictly convex, and bijective function on $\Rn$, and moreover that $\by \mapsto \nabla_{\by}K^{*}_{\epsilon}(\by,t)$ corresponds to the inverse of $\bx\mapsto\bu_{PM}(\bx,t,\epsilon)$, i.e., $\nabla_{\by}K^{*}_{\epsilon}(\bu_{PM}(\bx,t,\epsilon),t) = \bx$. Finally, as $\bx \mapsto K_{\epsilon}(\bx,t)$ is twice differentiable and strictly convex on $\Rn$, the inverse function theorem (Theorem 7, Appendix C, \cite{evans1998partial}) implies that $\by \mapsto \nabla_{\by}K^{*}_{\epsilon}(\by,t)$ is continuously differentiable on $\Rn$, whence $\by \mapsto K_\epsilon(\by,t)$.
\end{proof}
\section{Properties of MMSE and MAP estimators} \label{sec:properties_estimators}

In this section, we describe various properties of the Bayesian posterior mean estimate~\eqref{eq:pm_estimate} in terms of the data $\bx \in \Rn$, parameters $t>0$ and $\epsilon>0$, and the imaging regularization term $J$. Specifically, in Section~\ref{subsec:prop_topo-rep-mono}, we derive topological, representation, and monotonicity properties of the posterior mean estimate, which we use in Section~\ref{subsec:prop_bounds-est} to further derive an optimal upper bound on the mean square error and other bounds and limit properties of the posterior mean estimate. Finally, we describe the MAP and posterior mean estimates in terms of Bayes risks and their connections to HJ PDEs in Sect.~\ref{subsec:prop_breg_div}.

\subsection{Topological, representation, and monotonicity properties} \label{subsec:prop_topo-rep-mono}
This section describes the topological, representation, and monotonicity properties of the Bayesian posterior mean estimate~\eqref{eq:pm_estimate}, which are stated, respectively, in Propositions~\ref{prop:topo_properties}, \ref{prop:pm_rep_props}, and \ref{prop:mono_properties}.

The first result, Proposition~\ref{prop:topo_properties}, states that the posterior mean estimate belongs in the interior of the domain of $J$ for all data $\bx \in \Rn$ and parameters $t>0$ and $\epsilon>0$.
\begin{prop}[Topological properties] \label{prop:topo_properties}
Suppose the function $J$ satisfies assumptions (A1)-(A3), and let $\bx\in\Rn$, $t>0$, and $\epsilon>0$. Then the posterior mean estimate $\bu_{PM}(\bx,t,\epsilon)$ is contained in $\inter{\dom J}$. As a consequence, the subdifferential $\partial J(\bu_{PM}(\bx,t,\epsilon))$ is non-empty and 
\[
J(\bu_{PM}(\bx,t,\epsilon))\leqslant \expectationJ{J(\by)} < \epsilon\left(e^{S_\epsilon(\bx,t)/\epsilon} - 1\right) < +\infty.
\]
\end{prop}
\begin{proof}
See Appendix B.
\end{proof}

The second result, Proposition~\ref{prop:pm_rep_props}, gives representation formulas for the posterior mean estimate. In particular, when the regularization term $J$ satisfies assumptions (A1)-(A3) and $\dom J = \Rn$, the posterior mean estimate and MSE then satisfy representation formulas in terms of the mean minimal subgradient of $J$ given by $\expectationJ{\pi_{\partial J(\by)}(\boldsymbol{0})}$. These representation formulas are then used to show that when $\dom J \neq \Rn$, the posterior mean estimate can nonetheless be approximated using the first-order HJ PDE~\eqref{eq:hj_non_visc_pde} by smoothing the initial $J$ via a Moreau--Yosida approximation $S_0(\bx,\mu)$ for any $\mu >0$.

\begin{prop}[Representation properties]\label{prop:pm_rep_props}
Suppose the function $J$ satisfies assumptions (A1)-(A3), let $\bx\in\Rn$, $t>0$, and $\epsilon>0$, and let $(\bx,t) \mapsto S_0$ and $(\bx,t) \mapsto S_\epsilon$ denote the solutions to the first and second-order HJ PDEs~\eqref{eq:hj_non_visc_pde} and \eqref{eq:hj_visc_pde} with initial data $J$, respectively.
\begin{enumerate}

    \item[(i)] (Representation formulas) Suppose that $\dom J = \Rn$. Then the posterior mean estimate $\bu_{PM}(\bx,t,\epsilon)$ and the MSE $\expectationJ{\normsq{\by-\bu_{PM}(\bx,t,\epsilon)}}$ of the Bayesian posterior distribution~\eqref{eq:prob_dist_j} satisfy the representation formulas
\begin{equation} \label{eq:diff_pm_rep}
    \bu_{PM}(\bx,t,\epsilon) = \bx -t\expectationJ{\pi_{\partial J(\by)}(\boldsymbol{0})}
\end{equation}
and
\begin{equation} \label{eq:diff_var_rep}
    \expectationJ{\normsq{\by-\bu_{PM}(\bx,t,\epsilon)}} = nt\epsilon - t\expectationJ{\left \langle \pi_{\partial J(\by)}(\boldsymbol{0}), \by - \bu_{PM}(\bx,t,\epsilon) \right \rangle},
\end{equation}
with $\pi_{\partial J(\by)}(\boldsymbol{0}) = \nabla J(\by)$ when $J$ is continuously differentiable. In particular, the gradient and Laplacian of the solution $(\bx,t) \mapsto S_\epsilon(\bx,t)$ to the HJ PDE~\eqref{eq:hj_visc_pde} with initial data $J$ satisfy 
\[
\nabla_{\bx}S_\epsilon(\bx,t) = \expectationJ{\pi_{\partial J(\by)}(\boldsymbol{0})}
\]
and
\[
\Laplacian S_{\epsilon}(\bx,t) = \frac{1}{t\epsilon}\expectationJ{\left \langle \pi_{\partial J(\by)}(\boldsymbol{0}), \by - \bu_{PM}(\bx,t,\epsilon) \right \rangle},
\]

    \item[(ii)] (Limit formulas) Let $\{\mu_{k}\}_{k=1}^{+\infty}$ be a sequence of positive real numbers decreasing to zero. Then the gradient of the solution $(\bx,t) \mapsto S_\epsilon(\bx,t)$ to the HJ PDE~\eqref{eq:hj_visc_pde} with initial data $J$ satisfies the limit
\[
\nabla_{\bx} S_\epsilon(\bx,t) = \lim_{k\to +\infty}\left(\frac{\int_{\Rn} \nabla_{\by}S_0(\by,\mu_k)e^{-\left(\frac{1}{2t}\normsq{\bx-\by} + S_0(\by,\mu_k)\right)/\epsilon}\diff\by}{\int_{\Rn} e^{-\left(\frac{1}{2t}\normsq{\bx-\by} + S_0(\by,\mu_k)\right)/\epsilon}\diff\by}\right).
\]
As a consequence, the posterior mean estimate $\bu_{PM}(\bx,t,\epsilon)$ satisfies the limit
\begin{equation} \label{eq:approx_limit}
    \bu_{PM}(\bx,t,\epsilon) = \bx - t\lim_{k\to +\infty}\left(\frac{\int_{\Rn} \nabla_{\by}S_0(\by,\mu_k)e^{-\left(\frac{1}{2t}\normsq{\bx-\by} + S_0(\by,\mu_k)\right)/\epsilon}\diff\by}{\int_{\Rn} e^{-\left(\frac{1}{2t}\normsq{\bx-\by} + S_0(\by,\mu_k)\right)/\epsilon}\diff\by}\right).
\end{equation}
\end{enumerate}
\end{prop}

\begin{proof}
See Appendix C for the proof.
\end{proof}

\begin{rem}
Note that the representation formulas in Proposition~\ref{prop:pm_rep_props}(ii) may not hold if $\dom J \neq \Rn$. To see this, consider $J:\Rn\to\R\cup\{+\infty\}$ defined by
\begin{equation*}
J(\by) =
\begin{dcases}
0, & \textrm{if } \left\Vert \by \right\Vert_2 \leqslant 1,\\
+\infty, & \textrm{otherwise}.
\end{dcases}
\end{equation*}
Then the domain of $J$ is the unit sphere in $\Rn$, which is convex, and $J$ satisfies assumptions (A1)-(A3). The function $J$ is continuously differentiable in $\inter{\dom J}$, with $\nabla J(\by) = \boldsymbol{0}$ for every $\by\in\inter{\dom J}$. Clearly, $\expectationJ{\pi_{\partial J(\by)}(\boldsymbol{0})} = 0$. However, for every $\bx \neq \boldsymbol{0}$, the posterior mean estimate $\bu_{PM}(\bx,t,\epsilon) \neq \bx$. Hence, the representation formula~\eqref{eq:diff_pm_rep} does not hold in that case.
\end{rem}

The third result, Proposition~\ref{prop:mono_properties}, describes monotonicity properties of the posterior mean estimate, which in particular will be leveraged in the next subsection to derive an optimal upper bound for the MSE $\expectationJ{\normsq{\by-\bu_{PM}(\bx,t,\epsilon)}}$ and several estimates and limit results of $\bu_{PM}(\bx,t,\epsilon)$ in terms of the observed image $\bx$ and parameter $t>0$. Our proof of the following proposition, which is presented in Appendix~\ref{app:mono}, uses the properties of solutions to first-order HJ PDEs presented in Theorem~\ref{thm:e_u_nonvisc_hj} together with the representation formulas~\eqref{eq:diff_pm_rep} and \eqref{eq:diff_var_rep}.

\begin{prop}[Monotonicity property] \label{prop:mono_properties}
Suppose the function $J$ satisfies assumptions (A1)-(A3), and let $\bx\in\Rn$, $t>0$, and $\epsilon>0$. Suppose also that $J$ is strongly convex of parameter $m\geqslant0$ (with $m=0$ corresponding to the definition of convexity). Then for every $\by_0 \in \dom\partial J$,
\begin{equation}  \label{eq:monotonicity_prop}
\begin{alignedat}{1}
\left(\frac{1+mt}{t}\right)\expectationJ{\normsq{\by-\by_{0}}} &\leqslant \expectationJ{\left\langle \left(\frac{\by - \bx}{t} + \pi_{\partial J(\by)}(\boldsymbol{0})\right)-\left(\frac{\by_0 - \bx}{t} + \pi_{\partial J(\by_0)}(\boldsymbol{0})\right),\by-\by_{0}\right\rangle}
\\
&\leqslant n\epsilon -\left\langle \left(\frac{\by_0 - \bx}{t} + \pi_{\partial J(\by_0)}(\boldsymbol{0})\right), \bu_{PM}(\bx,t,\epsilon) - \by_0 \right\rangle.
\end{alignedat}
\end{equation}
\end{prop}

\begin{proof}
See Appendix~\ref{app:mono} for the proof.
\end{proof}

As a corollary of this result, we now show that the mean minimal subgradient $\expectationJ{\pi_{\partial J(\by)}(\boldsymbol{0})}$ is finite; this fact will be used later in Subsection~\ref{subsec:prop_breg_div} for proving Theorem~\ref{thm:bregman_div}(i).

\begin{cor}\label{cor:mmsp}
Suppose $J$ satisfies assumptions (A1)-(A3). For every $\bx\in\Rn$, $t>0$, and $\epsilon>0$, the mean minimal subgradient $\expectationJ{\pi_{\partial J(\by)}(\boldsymbol{0})}$ is finite.
\end{cor}
\begin{proof}
Let $\by_0$ be any element of $\inter{\dom J}$ different from $\bu_{PM}(\bx,t,\epsilon)$; such an element exists because $\inter{\dom J} \neq \varnothing$ by assumption (A2). Consider the scalar product
\begin{equation*}
\begin{alignedat}{1}
\left\langle \frac{\by - \bx}{t} + \pi_{\partial J(\by)}(\boldsymbol{0}), \by_0 - \bu_{PM}(\bx,t,\epsilon)  \right\rangle = &\left\langle \left(\frac{\by - \bx}{t} + \pi_{\partial J(\by)}(\boldsymbol{0})\right), \by - \bu_{PM}(\bx,t,\epsilon)  \right\rangle
\\
& - \left\langle \left(\frac{\by - \bx}{t} + \pi_{\partial J(\by)}(\boldsymbol{0})\right), \by - \by_0 \right\rangle.
\end{alignedat}
\end{equation*}
Take the expectation $\mathbb{E}_J\left[\cdot\right]$ and use the monotonicity property~\eqref{eq:monotonicity_prop} to get the inequality
\[
-n\epsilon \leqslant \left\langle \left(\frac{\bu_{PM}(\bx,t,\epsilon) - \bx}{t} + \expectationJ{\pi_{\partial J(\by)}(\boldsymbol{0})}\right) - \left(\frac{\by_0 - \bx}{t} + \pi_{\partial J(\by_0)}(\boldsymbol{0})\right), \bu_{PM}(\bx,t,\epsilon) - \by_0 \right\rangle \leqslant n\epsilon.
\]
The scalar product in the equation above is therefore finite, which implies that the mean minimal subgradient $\expectationJ{\pi_{\partial J(\by)}(\boldsymbol{0})}$ in the scalar product is also finite.
\end{proof}

\subsection{Bound and limit properties} \label{subsec:prop_bounds-est}

In this section, we derive an optimal bound on the MSE $\expectationJ{\normsq{\by-\bu_{PM}(\bx,t,\epsilon)}}$, various bounds on the posterior mean estimate $\bu_{PM}(\bx,t,\epsilon)$, and limiting results of the posterior mean estimate in terms of the parameters $t$.

\begin{prop}[Bounds and limit properties] \label{prop:bounds_limits}
Suppose $J$ satisfies assumptions (A1)-(A3), and suppose that it is strongly convex of parameter $m \geqslant 0$ (with $m = 0$ corresponding to the definition of convexity).

\begin{itemize}
    \item[(i)] For every $\bx \in \Rn$, $t>0$, and $\epsilon>0$, the MSE $\expectationJ{\normsq{\by-\bu_{PM}(\bx,t,\epsilon)}}$ of the Bayesian posterior distribution~\eqref{eq:prob_dist_j} satisfies the upper bound 
    \begin{equation} \label{eq:mono_var_ub}
        \expectationJ{\normsq{\by-\bu_{PM}(\bx,t,\epsilon)}} \leqslant \frac{nt\epsilon}{1+mt}.
    \end{equation}

    \item[(ii)] For every $\bx \in \Rn$, $t>0$, and $\epsilon>0$, the square of the Euclidean norm between the posterior mean estimate and the MAP estimate satisfies the upper bound
    \begin{equation} \label{eq:difference_pm_map}
        \normsq{\bu_{MAP}(\bx,t)-\bu_{PM}(\bx,t,\epsilon)} \leqslant \frac{nt\epsilon}{1+mt}.
    \end{equation}
    
    \item[(iii)] The posterior mean estimate is monotone and non-expensive, that is, for every $\bx$, $\boldsymbol{d} \in \Rn$, $t>0$, and $\epsilon>0$,
    \begin{equation} \label{eq:upm_monotone}
        \left\langle \bu_{PM}(\bx + \boldsymbol{d},t,\epsilon) - \bu_{PM}(\bx,t,\epsilon), \boldsymbol{d} \right\rangle \geqslant 0
    \end{equation}
    and
    \begin{equation} \label{eq:lipschitz_cts}
         \left\Vert\bu_{PM}(\bx + \boldsymbol{d},t,\epsilon) -\bu_{PM}(\bx,t,\epsilon)\right\Vert_2  \leqslant \left\Vert \boldsymbol{d}\right\Vert_2.
    \end{equation}
    \item[(iv)] Let $\{t_k\}_{k=1}^{+\infty}$ be a sequence of positive real numbers converging to $0$ and let $\{\boldsymbol{d}_k\}_{k=1}^{+\infty}$ be a sequence of elements of $\Rn$ converging to $\boldsymbol{d} \in \Rn$. Then for every $\bx \in \dom J$ and $\epsilon>0$, the pointwise limit of $\bu_{PM}(\bx + t_k\boldsymbol{d}_k,t_k,\epsilon)$ as $k\to +\infty$ exists and satisfies
    \[
        \lim_{k \to +\infty} \bu_{PM}(\bx + t_k\boldsymbol{d}_k,t_k,\epsilon) = \bx.
    \]
\end{itemize}
\end{prop}

\begin{proof}
Proof of (i):  Since $\bu_{PM}(\bx,t,\epsilon) \in \inter{\dom J}$ by Proposition~\ref{prop:topo_properties} and $\inter{\dom J} \subset \dom \partial J$ (see Definition~\ref{def:subgrad}), we can set $\by_0 = \bu_{PM}(\bx,t,\epsilon)$ in the monotonicity inequality~\eqref{eq:monotonicity_prop} in Proposition~\ref{prop:mono_properties}(i) and rearrange to get the upper bound~\eqref{eq:mono_var_ub}.

Proof of (ii):  Note that for every $\by_0 \in \dom\partial J$, the monotonicity inequality~\eqref{eq:monotonicity_prop} in Proposition~\ref{prop:mono_properties} yields
\[
\expectationJ{\left\langle \left(\frac{\by - \bx}{t} + \pi_{\partial J(\by)}(\boldsymbol{0})\right),\by-\by_{0}\right\rangle} \leqslant n\epsilon.
\]
Choose $\by_0 = \bu_{MAP}(\bx,t)$, which for every $\bx$ and $t>0$ is always an element of $\dom \partial J$ and also satisfies the inclusion $\left(\frac{\bx-\bu_{MAP}(\bx,t)}{t}\right) \in \partial J(\bu_{MAP}(\bx,t))$ by part (ii) of Theorem~\ref{thm:e_u_nonvisc_hj}. Hence the monotonicity of the subdifferential of $\by \mapsto \frac{1}{2t}\normsq{\bx-\by} + J(\by)$ and strong convexity of $J$ of parameter $m\geqslant0$ implies
\[
\left(\frac{1+mt}{t}\right)\normsq{\by-\bu_{MAP}(\bx,t)} \leqslant \left\langle \left(\frac{\bx - \by}{t} + \pi_{\partial J(\by)}(\boldsymbol{0})\right),\by-\bu_{MAP}(\bx,t)\right\rangle.
\]
Combine these inequalities to get $\expectationJ{\normsq{\by-\bu_{MAP}(\bx,t)}} \leqslant \frac{nt\epsilon}{1+mt}$, and use the convexity of the Euclidean norm to get inequality~\eqref{eq:difference_pm_map}.

Proof of (iii): The convexity of $\bx \mapsto K_{\epsilon}(\bx,t)$ by Theorem~\ref{thm:visc_hj_eqns}(ii)(d) and $\nabla_{\bx}K_{\epsilon}(\bx,t) = \bu_{PM}(\bx,t,\epsilon)$ implies the monotone property~\eqref{eq:upm_monotone} (see definition~\ref{def:subgrad}, equation~\eqref{eq:monotone_mapping}, and \cite{rockafellar1970convex}, page 240 and Corollary 31.5.2). Since both functions $\bx \mapsto S_\epsilon(\bx,t)$ and $\bx \mapsto \frac{1}{2}\normsq{\bx} - tS_\epsilon(\bx,t)$ are convex by Theorem~\ref{thm:visc_hj_eqns}(ii)(a) and (d), the gradient of the function $\bx \mapsto \frac{1}{2}\normsq{\bx} - tS_\epsilon(\bx,t)$, whose value is the posterior mean estimate $\bu_{PM}(\bx,t,\epsilon)$ by Theorem~\ref{thm:visc_hj_eqns}(iii), is Lipschitz continuous with unit constant (see \cite{zhou2018fenchel} for a simple proof), that is,
\[
\left\Vert \left(\bx+\boldsymbol{d} - t\nabla_{\bx}S_\epsilon(\bx+\boldsymbol{d},t)\right) - \left(\bx - t\nabla_{\bx}S_\epsilon(\bx,t)\right) \right\Vert_2 \equiv \left\Vert\bu_{PM}(\bx+\boldsymbol{d},t,\epsilon) - \bu_{PM}(\bx,t,\epsilon)\right\Vert_2 \leqslant \left\Vert\boldsymbol{d}\right\Vert_2,
\]
which proves the non-expensive inequality~\eqref{eq:lipschitz_cts}.

Proof of (iv): Inequality~\eqref{eq:difference_pm_map} and the triangle inequality imply
\begin{equation*}
\left\Vert (\bx + t_k\boldsymbol{d}_k)- \bu_{PM}(\bx + t_k\boldsymbol{d}_k,t_k,\epsilon)\right\Vert_{2} \leqslant \left\Vert (\bx + t_k\boldsymbol{d}_k) - \bu_{MAP}(\bx + t_k\boldsymbol{d}_k,t_k)\right\Vert_{2} + \sqrt{\frac{nt_k\epsilon}{1+mt}}.
\end{equation*}
The limit $\lim_{k \to +\infty} \bu_{PM}(\bx + t_k\boldsymbol{d}_k,t_k,\epsilon) = \bx$ then follows by Theorem~\ref{thm:e_u_nonvisc_hj}(i).
\end{proof}

\begin{rem}
The upper bound for the MSE in \eqref{eq:mono_var_ub} is optimal; as shown in Example~\ref{example:tikhonov} it is attained for the quadratic term $J(\bx) = \frac{m}{2}\normsq{\bx}$. 
\end{rem}

\subsection{Bayesian risks and Hamilton--Jacobi partial differential equations} \label{subsec:prop_breg_div}

In this section, we will consider the Bayesian risk associated to the Bregman loss function
\begin{equation}\label{eq:breg2}
\by \mapsto D_{\Phi_{J}}(\bu,\varphi_{J}(\by|\bx,t)),
\end{equation}
where
\begin{equation*}
    \Rn \times \Rn \times (0,+\infty) \ni (\by,\bx,t) \mapsto \Phi_J(\by|\bx,t) = \frac{1}{2t}\normsq{\bx-\by} + J(\by),
\end{equation*}
which is up to a constant the negative logarithm of the posterior distribution~\eqref{eq:prob_dist_j}, and
\begin{equation*}
    \Rn \times \Rn \times (0,+\infty) \ni (\by,\bx,t) \mapsto \varphi_J(\by|\bx,t) = \left(\frac{\by-\bx}{t}\right) + \pi_{\partial J(\by)}(\boldsymbol{0}),
\end{equation*}
which is a subgradient of the function $\by \mapsto \Phi_{J}(\by|\bx,t)$. The corresponding Bayesian risk to the posterior distribution~\eqref{eq:prob_dist_j} correspond to the expected value $\expectationJ{D_{\Phi_{J}}(\bu,\varphi_{J}(\by|\bx,t))}$. We refer the reader to \cite{banerjee2005optimality} and \cite{keener2011theoretical} for discussions on Bregman loss functions and Bayesian estimation theory.

Recent work by \cite{burger2014maximum} has shown that the MAP estimate~\eqref{eq:minimizer_prob} corresponds to the Bayes estimator associated to the Bregman loss function~\eqref{eq:breg2} when the regularization term $J$ is convex and uniformly Lipschitz continuous on $\Rn$. This was later extended by \cite{burger2016bregman} to posterior distributions with non-Gaussian fidelity term, and later studied from the point of view from differential geometry in \cite{pereyra2019revisiting} and also derived for posterior distributions that are strongly log-concave and sufficiently smooth. Here, we will use the connections between maximum a posteriori and posterior mean estimates and Hamilton--Jacobi equations derived in Section~\ref{sec:connections} to show that when the regularization term $J$ is convex on $\Rn$, then the MAP estimate $\bu_{MAP}(\bx,t)$ minimizes in expectation the Bregman loss function~\eqref{eq:breg2}. Thus, under the assumption of a Gaussian data fidelity term, this result generalizes the result from \citet{burger2014maximum} (Theorem 1) by removing the uniformly Lipschitz continuity assumption on $J$. Moreover, we also show that when $\dom J \neq \Rn$, there still exists a Bayes estimator. A similar result was established in \citet{pereyra2019revisiting} (see Theorem 4 and section 5.3), where $J$ was assumed to be thrice differentiable and under strong convexity assumptions on the posterior distribution. In contrast, our results only need that $J$ satisfies assumptions (A1)-(A3). The results rely on the monotonicity property~\eqref{eq:monotonicity_prop} and finiteness of the mean minimal subgradient $\expectationJ{\pi_\partial(\by)(\boldsymbol{0})}$ as shown in Corollary~\ref{cor:mmsp}.

\begin{thm}[Bregman divergences] \label{thm:bregman_div}
Suppose the function $J$ satisfies assumptions (A1)-(A3), and let $\bx\in\Rn$, $t>0,$ and $\epsilon>0$.
\begin{enumerate}

\item[(i)] The mean Bregman loss function  $\dom J\ni\bu\mapsto\expectationJ{D_{\Phi_{J}}(\bu,\varphi_{J}(\by|\bx,t))} \in \R$ has a unique minimizer $\bar{\bu}\in\dom\partial J$ that satisfies the inclusion
\begin{equation} \label{eq:subdiff_rep1}
\left(\frac{\bx-\bar{\bu}}{t}\right) \in \partial J(\bar{\bu}) + \left( \nabla_{\bx}S_\epsilon(\bx,t) - \expectationJ{\pi_{\partial J(\by)}(\boldsymbol{0})}\right),
\end{equation}
where addition in \eqref{eq:subdiff_rep1} is taken in the sense of sets.

\item[(ii)] If $J$ is finite everywhere on $\Rn$, i.e., $\dom J = \Rn$, then the MAP estimate $\bu_{MAP}(\bx,t)$ is the unique global minimizer of the Bregman loss function $\Rn \ni \bu\mapsto\expectationJ{D_{\Phi_{J}}(\bu,\varphi_{J}(\by|\bx,t))} \in \R$, i.e.,
\begin{equation} \label{eq:map_est_bregman}
    \bu_{MAP}(\bx,t) = \argmin_{\bu \in \Rn} \expectationJ{D_{\Phi_{J}}(\bu,\varphi_{J}(\by,\bx,t))}
\end{equation}
\end{enumerate}
\end{thm}

\begin{proof}
See Appendix E for the proof.
\end{proof}
\section{Conclusion} \label{sec:conclusion}

In this paper, we presented original connections between Hamilton--Jacobi partial differential equations and a broad class of Bayesian posterior mean estimators with Gaussian data fidelity term and log-concave prior relevant to image denoising problems. We derived representation formulas for the posterior mean estimate $\bu_{PM}(\bx,t,\epsilon)$ in terms of the spatial gradient of the solution to a viscous HJ PDE with initial data corresponding to the convex regularization term $J$. We used these connections that the posterior mean estimate can be expressed through the gradient of the solution to a first-order HJ PDE with twice continuously differentiable convex initial data. The connections between HJ PDEs and Bayesian posterior mean estimators were further used to establish several topological, representation, and monotonicity properties of posterior mean estimates. These properties were then used to derive an optimal upper bound for the mean squared error $\expectationJ{\normsq{\by-\bu_{PM}(\bx,t,\epsilon)}}$, several estimates on the MAP and posterior mean estimates, and the behavior of the posterior mean estimate $\bu_{PM}(\bx,t,\epsilon)$ in the limit $t \to 0$. Finally, we used the connections between both MAP and posterior mean estimates and HJ PDEs to show that the MAP estimate~\eqref{eq:minimizer_prob} corresponds to the Bayes estimator of the Bayesian risk~\eqref{eq:breg2} whenever the regularization term $J$ is convex on $\Rn$ and the data fidelity term is Gaussian. We also show that when $\dom J \neq \Rn$, the Bayesian risk~\eqref{eq:breg2} has still a Bayes estimator that is described in terms of the solution to both the first-order HJ PDE~\eqref{thm:e_u_nonvisc_hj} and the viscous HJ PDE~\eqref{thm:visc_hj_eqns}.

We wish to note that in addition to its relevance to image denoising problems, the viscous HJ PDE~\eqref{intro:pde} has recently received some attention in the deep learning literature, where its solution $\bx \mapsto S_\epsilon(\bx,t)$ is known as the local entropy loss function and is a loss regularization effective at training deep networks \cite{chaudhari2018deep,chaudhari2019entropy,garcia2019variational,vidal2017mathematics}. While this paper focuses on HJ PDEs and Bayesian estimators in imaging sciences, the results in this paper may be relevant to the deep learning literature and may give new theoretical understandings of the local entropy loss function in terms of the data $\bx$ and parameters $t$ and $\epsilon$.

The results presented in this work crucially depend on the data fidelity term being Gaussian and the generalized prior distribution $\by \mapsto e^{-J(\by)}$ being log-concave. This paper did not consider non-Gaussian data fidelity terms with log-concave priors, or non-additive noise models \cite{boncelet2009image,boyat2015review}.

\appendix

\section{Proof of Theorem~\ref{thm:visc_hj_eqns}} \label{app:A}
To prove Theorem~\ref{thm:visc_hj_eqns}, we will first use the following lemma, which characterizes the partition function \eqref{eq:partition_function} in terms of the solution to a Cauchy problem involving the heat equation with initial data $J \in \gmRn$. This connection will imply parts (i) and (ii)(a)-(d) of Theorem~\ref{thm:visc_hj_eqns}.

\begin{lem}[The heat equation with initial data in $\gmRn$] \label{lem:heat_eqn}
Suppose the function $J$ satisfies assumptions (A1)-(A3).
\begin{enumerate}
\item[(i)] For every $\epsilon>0$, the function $w_\epsilon\colon\Rn\times[0,+\infty)\to (0,1]$ defined by
\begin{equation}
w_{\epsilon}(\bx,t)\coloneqq \frac{1}{(2\pi t \epsilon)^{n/2}}Z_J(\bx,t,\epsilon) = \frac{1}{\left(2\pi t\epsilon\right)^{n/2}}\int_{\Rn}e^{-\left(\frac{1}{2t}\left\Vert \bx-\by\right\Vert _{2}^{2}+J(\by)\right)/\epsilon}\diff\by\label{eq:heat_w}
\end{equation}
is the unique smooth solution to the Cauchy problem
\begin{equation}
\begin{dcases}
\frac{\partial w_{\epsilon}}{\partial t}(\bx,t)=\frac{\epsilon}{2}\Laplacian w_{\epsilon}(\bx,t) & \mbox{in }\Rn\times(0,+\infty),\\
w_{\epsilon}(\bx,0)=e^{-J(\bx)/\epsilon} & \mbox{in }\Rn.
\end{dcases}\label{eq:cauchy_heat_prob}
\end{equation}
In addition, the domain of integration of the integral \eqref{eq:heat_w} can be taken to be $\dom J$ or, up to a set of Lebesgue measure zero, $\inter(\dom J)$
or $\dom\partial J$. Furthermore, for every $\bx\in\dom J$ and $\epsilon > 0$, except possibly at the boundary points $\bx\in(\dom J)\backslash(\inter(\dom J))$ if such points exist, the pointwise limit of $w_\epsilon(\bx,t)$ as $t \to 0$ exists and satisfies
\begin{equation*}
\lim_{\substack{t\to0\\ t>0}}w_{\epsilon}(\bx,t)=e^{-J(\bx)/\epsilon},
\end{equation*}
with the limit equal to $0$ whenever $\bx\notin\dom J$. If $\bx\in(\dom J)\backslash(\inter(\dom J))$, then we may only conclude 
\[
\limsup_{\substack{t\to0\\t>0}}w_{\epsilon}(\bx,t)\leqslant e^{-J(\bx)/\epsilon}
\]
and
\[
\liminf_{\substack{t\to0\\t>0}}w_{\epsilon}(\bx,t)\geqslant e^{-J(\bx)/\epsilon}\left(\frac{1}{(2\pi\epsilon)^{n/2}}\int_{\dom J}e^{-\frac{1}{2\epsilon}\left\Vert \bx-\by\right\Vert _{2}^{2}}\diff\by\right).
\]

\item[(ii)]  (Log-concavity and monotonicity properties).
\begin{enumerate}
\item The function $\Rn\times(0,+\infty)\ni(\bx,t)\mapsto t^{n/2}w_{\epsilon}(\bx,t)$ is jointly log-concave.
\item The function $(0,+\infty)\ni t\mapsto t^{n/2}w_{\epsilon}(\bx,t)$ is strictly monotone increasing.
\item The function $(0,+\infty)\ni\epsilon\mapsto\epsilon^{n/2}w_{\epsilon}(\bx,t)$ is strictly monotone increasing.
\item The function $\Rn\ni\bx\mapsto e^{\frac{1}{2t\epsilon}\normsq{\bx}}w_{\epsilon}(\bx,t)$ is strictly log-convex.
\end{enumerate}
\end{enumerate}
\end{lem}
The proof of (i) follows from classical PDEs arguments for the Cauchy problem \eqref{eq:cauchy_heat_prob} tailored to the initial data $(\bx,\epsilon)\mapsto e^{-J(\bx)/\epsilon}$ with $J$ satisfying assumptions (A1)-(A3), and the proof of log-concavity and monotonicity (ii)(a)-(d) follows from the Pr\'ekopa--Leindler and H\"older's inequalities \citep{leindler1972certain, prekopa1971logarithmic,folland2013real}; we present the details below.

\begin{proof}
Proof of Lemma~\ref{lem:heat_eqn} (i): By assumptions (A1) and (A2), there exists a point $\by_{0}\in\inter(\dom J)$ and a number $\delta>0$ such that the open ball $B_{\delta}(\by_{0})$ is contained in $\inter(\dom J)$ and $e^{-J(\by)/\epsilon}>0$ whenever $\by\in B_{\delta}(\by_{0})$. Since assumption (A3) yields $e^{-J(\by)/\epsilon}\leqslant1$ for every $\by\in\Rn$, these observations imply that the Cauchy problem \eqref{eq:cauchy_heat_prob} has a unique, smooth solution defined by equation \eqref{eq:heat_w}, with $0<w_\epsilon(\bx,t)\leqslant1$ for every $\bx\in\Rn$, $t>0$, and $\epsilon>0$ (see \citet{widder1976heat}, Theorem 1 in Chapter VII for global existence and smoothness properties of solutions, and Theorem 2.2 in Chapter VIII for uniqueness of solutions, and note that the results in \cite{widder1976heat} are for $n=1$ but can be extended without difficulty to $n>1$). In addition, as $e^{-J(\by)/\epsilon}=0$ for every $\by\notin\dom J$, the domain of integration of \eqref{eq:heat_w} can be taken to be $\dom J$, as the boundary points of the domain of $J$ is a set of Lebesgue measure zero relative to $\Rn$ (see definition \ref{def:convex_sets}), the domain of integration can be further taken to be $\inter(\dom J)$ or $\dom\partial J$.

Now, we will use Fatou's lemma (\citep{folland2013real}, Lemma 2.18) to compute bounds for the two limits $\limsup_{\substack{t\to0\\t>0}}w_{\epsilon}(\bx,t)$ and $\liminf_{\substack{t\to0\\t>0}}w_{\epsilon}(\bx,t)$ for every $\bx\in\Rn$ and $\epsilon>0$. First, note that the change of variables $\frac{\bx-\by}{\sqrt{t}}\mapsto\by$ in \eqref{eq:heat_w} yields
\[
w_{\epsilon}(\bx,t)=\frac{1}{(2\pi\epsilon)^{n/2}}\int_{\Rn}e^{-\left(\frac{1}{2}\left\Vert \by\right\Vert _{2}^{2}+J(\bx-\sqrt{t}\by)\right)/\epsilon}\diff\by,
\]
so that the function
\[
\by\mapsto e^{-\frac{1}{2\epsilon}\left\Vert y\right\Vert _{2}^{2}}\left(1-e^{-J(\bx-\sqrt{t}\by)/\epsilon}\right)
\]
is non-negative. The reverse Fatou's lemma therefore also applies to this function, and hence
\[
\limsup_{\substack{t\to0\\t>0}}w_{\epsilon}(\bx,t)\leqslant\frac{1}{(2\pi\epsilon)^{n/2}}\int_{\Rn}e^{-\frac{1}{2\epsilon}\left\Vert \by\right\Vert _{2}^{2}}\left(\limsup_{\substack{t\to0\\t>0}}e^{-J(\bx-\sqrt{t}\by)/\epsilon}\right)\diff\by.
\]
Using the lower semicontinuity of $J$, the limit inside the integral satisfies
\begin{eqnarray*}
\limsup_{\substack{t\to0\\t>0}}e^{-J(\bx-\sqrt{t}\by)/\epsilon} & = & e^{\limsup_{\substack{t\to0\\t>0}}-J(\bx-\sqrt{t}\by)/\epsilon}\\
& = & e^{-\liminf_{\substack{t\to0\\t>0}}J(\bx-\sqrt{t}\by)/\epsilon}\\
& \leqslant & e^{-J(\bx)/\epsilon},\quad\forall\by\in\Rn,
\end{eqnarray*}
and therefore $\limsup_{\substack{t\to0\\t>0}}w_{\epsilon}(\bx,t)\leqslant e^{-J(\bx)/\epsilon}$ for every $\bx\in\Rn$. If $\bx\notin\dom J$, then
\[
0\leqslant\liminf_{\substack{t\to0\\t>0}}e^{-J(\bx-\sqrt{t}\by)/\epsilon}\leqslant\limsup_{\substack{t\to0\\t>0}}e^{-J(\bx-\sqrt{t}\by)/\epsilon}\leqslant0,
\]
which implies $\lim_{\substack{t\to0\\t>0}}w_{\epsilon}(\bx,t)=0$ for every $\bx\notin\dom J$. Suppose now that $\bx\in\dom J$. By Fatou's lemma, 
\[
\liminf_{\substack{t\to0\\t>0}}w_{\epsilon}(\bx,t)\geqslant\frac{1}{(2\pi\epsilon)^{n/2}}\int_{\Rn}e^{-\frac{1}{2\epsilon}\left\Vert \by\right\Vert _{2}^{2}}\left(\liminf_{\substack{t\to0\\t>0}}e^{-J(\bx-\sqrt{t}\by)/\epsilon}\right)\diff\by.
\]
If $\bx\in\inter(\dom J)$, then by continuity $\lim_{\substack{t\to0\\t>0}}J(\bx-\sqrt{t}\by)=J(\bx)$ for every $\by\in\Rn$, and so $\liminf_{\substack{t\to0\\t>0}}w_{\epsilon}(\bx,t)\geqslant e^{-J(\bx)/\epsilon}$. Combined with the $\limsup$, we find $\lim_{\substack{t\to0\\t>0}}w_{\epsilon}(\bx,t)=e^{-J(\bx)/\epsilon}$ for every $\bx \in \inter(\dom J)$ and $\bx \notin \dom J$. If $\bx\in(\dom J)\backslash\inter(\dom J)$, if any such point exists, and $0<t\leqslant1$, then convexity of $J$ implies
\begin{align*}
J(\bx-\sqrt{t}\by) & =J((1-\sqrt{t})\bx+\sqrt{t}(\bx-\by))\\
& \leqslant(1-\sqrt{t})J(\bx)+\sqrt{t}J(\bx-\by),
\end{align*}
and hence for any $0<t\leqslant1$,
\begin{align*}
w_{\epsilon}(\bx,t) & \geqslant e^{-(1-\sqrt{t})J(\bx)/\epsilon}\frac{1}{(2\pi\epsilon)^{n/2}}\int_{\Rn}e^{-\frac{1}{2\epsilon}\left\Vert \by\right\Vert _{2}^{2}}e^{-\sqrt{t}J(\bx-\by)/\epsilon}\diff\by,\\
& =e^{-(1-\sqrt{t})J(\bx)/\epsilon}\frac{1}{(2\pi\epsilon)^{n/2}}\int_{\Rn}e^{-\frac{1}{2\epsilon}\left\Vert \bx-\by\right\Vert _{2}^{2}}e^{-\sqrt{t}J(\by)/\epsilon}\diff\by.
\end{align*}
Since $\liminf_{\substack{t\to0\\t>0}}e^{-\sqrt{t}J(\by)/\epsilon} = 1$ for every $\by\in\dom J$ and is equal to $0$ for every $\by\notin\dom J$, Fatou's lemma yields
\[
\liminf_{\substack{t\to0\\t>0}}w_{\epsilon}(\bx,t)\geqslant e^{-J(\bx)/\epsilon}\left(\frac{1}{(2\pi\epsilon)^{n/2}}\int_{\dom J}e^{-\frac{1}{2\epsilon}\left\Vert \bx-\by\right\Vert _{2}^{2}}\diff\by\right)
\]
for every $\bx\in(\dom J)\backslash\inter(\dom J)$, if any such point exists.

Proof of Lemma~\ref{lem:heat_eqn} (ii)(a): The log-concavity property will be shown using the Pr\'ekopa--Leindler inequality.
\begin{thm}\label{thm:prekopa}
[Pr\'ekopa--Leindler inequality \citep{leindler1972certain, prekopa1971logarithmic}]
Let $f$, $g,$ and $h$ be non-negative real-valued and Borel measurable functions on $\Rn$, and suppose
\[
h(\lambda\by_{1}+(1-\lambda)\by_{2})\geqslant f(\by_{1})^{\lambda}g(\by_{2})^{(1-\lambda)}
\]
for every $\by_{1},$ $\by_{2}\in\Rn$ and $\lambda\in(0,1)$. Then
\[
\int_{\Rn}h(\by)\diff\by\geqslant\left(\int_{\Rn}f(\by)\diff\by\right)^{\lambda}\left(\int_{\Rn}g(\by)\diff\by\right)^{(1-\lambda)}.
\]
\end{thm}
Let $\epsilon>0$, $\lambda\in(0,1)$, $\bx=\lambda\bx_{1}+(1-\lambda)\bx_{2}$, $\by=\lambda\by_{1}+(1-\lambda)\by_{2}$, and $t=\lambda t_{1}+(1-\lambda)t_{2}$ for any $\bx_{1},\,\bx_{2},\,\by_{1},\,\by_{2}\in\Rn$ and $t_{1},\,t_{2}\in(0,+\infty)$. The joint convexity of the function $\Rn\times(0,+\infty)\ni(\boldsymbol{z},t)\mapsto\frac{1}{2t}\left\Vert \boldsymbol{z}\right\Vert _{2}^{2}$ and convexity of $J$ imply
\[
\frac{1}{2t}\left\Vert \bx-\by\right\Vert _{2}^{2}+J(\by) \leqslant \frac{\lambda}{2t_{1}}\left\Vert \bx_{1}-\by_{1}\right\Vert _{2}^{2}+\frac{(1-\lambda)}{2t_{2}}\left\Vert \bx_{2}-\by_{2}\right\Vert _{2}^{2}+\lambda J(\by_{1})+(1-\lambda)J(\by_{2}),
\]
This gives
\[
\frac{e^{-\left(\frac{1}{2t}\left\Vert \bx-\by\right\Vert _{2}^{2}+J(\by)\right)/\epsilon}}{(2\pi\epsilon)^{n/2}}\geqslant\left(\frac{e^{-\left(\frac{1}{2t_{1}}\left\Vert \bx_{1}-\by_{1}\right\Vert _{2}^{2}+J(\by_{1})\right)/\epsilon}}{(2\pi\epsilon)^{n/2}}\right)^{\lambda}\left(\frac{e^{-\left(\frac{1}{2t_{2}}\left\Vert \bx_{2}-\by_{2}\right\Vert _{2}^{2}+J(\by_{2})\right)/\epsilon}}{(2\pi\epsilon)^{n/2}}\right)^{1-\lambda}.
\]
Applying the Pr\'ekopa--Leindler inequality with
\[
h(\by)=\frac{e^{-\left(\frac{1}{2t}\left\Vert \bx-\by\right\Vert _{2}^{2}+J(\by)\right)/\epsilon}}{(2\pi\epsilon)^{n/2}},
\]
\[
f(\by)=\frac{e^{-\left(\frac{1}{2t_{1}}\left\Vert \bx_{1}-\by\right\Vert _{2}^{2}+J(\by)\right)/\epsilon}}{(2\pi\epsilon)^{n/2}},
\]
and 
\[
g(\by)=\frac{e^{-\left(\frac{1}{2t_{2}}\left\Vert \bx_{2}-\by\right\Vert _{2}^{2}+J(\by)\right)/\epsilon}}{(2\pi\epsilon)^{n/2}},
\]
and using the definition \eqref{eq:heat_w} of $w_{\epsilon}(\bx,t)$, we get
\[
t^{n/2}w_{\epsilon}(\bx,t)\geqslant\left(t_{1}^{n/2}w_{\epsilon}(\bx_{1},t_{1})\right)^{\lambda}\left(t_{2}^{n/2}w_{\epsilon}(\bx_{2},t_{2})\right)^{(1-\lambda)},
\]
As a result, the function $(\bx,t)\mapsto t^{n/2}w_{\epsilon}(\bx,t)$ is jointly log-concave on $\Rn\times(0,+\infty)$.

Proof of Lemma~\ref{lem:heat_eqn} (ii)(b): Since $t\mapsto\frac{1}{t}$ is strictly monotone decreasing on $(0,+\infty)$, for any $\bx\in\Rn$, $\epsilon>0$, and $0<t_1<t_2$,
\[
\frac{e^{-\left(\frac{1}{2t_1}\left\Vert \bx-\by\right\Vert _{2}^{2}+J(\by)\right)/\epsilon}}{(2\pi\epsilon)^{n/2}}<\frac{e^{-\left(\frac{1}{2t_2}\left\Vert \bx-\by\right\Vert _{2}^{2}+J(\by)\right)/\epsilon}}{(2\pi\epsilon)^{n/2}}
\]
whenever $\bx\neq\by$. Integrating both sides of the inequality with respect to $\by$ over $\dom J$ yields
\[
\frac{1}{(2\pi\epsilon)^{n/2}}\int_{\dom J}e^{-\left(\frac{1}{2t_1}\left\Vert \bx-\by\right\Vert _{2}^{2}+J(\by)\right)/\epsilon}\diff\by<\frac{1}{(2\pi\epsilon)^{n/2}}\int_{\dom J}e^{-\left(\frac{1}{2t_2}\left\Vert \bx-\by\right\Vert _{2}^{2}+J(\by)\right)/\epsilon}\diff\by,
\]
As a result, the function $t\mapsto t^{n/2}w_{\epsilon}(\bx,t)$ is strictly monotone increasing on $(0,+\infty)$. 

Proof of Lemma~\ref{lem:heat_eqn} (ii)(c): Since $\epsilon\mapsto\frac{1}{\epsilon}$ is strictly monotone decreasing on $(0,+\infty)$ and $\by\mapsto J(\by)$ is non-negative by assumption (A3), for any $\bx\in\Rn$, $t>0$, and $0<\epsilon_1<\epsilon_2$,
\[
e^{-\left(\frac{1}{2t}\left\Vert \bx-\by\right\Vert _{2}^{2}+J(\by)\right)/\epsilon_1}<e^{-\left(\frac{1}{2t}\left\Vert \bx-\by\right\Vert _{2}^{2}+J(\by)\right)/\epsilon_2}
\]
whenever $\bx\neq\by$. Integrating both sides of the inequality with respect to $\by$ over $\dom J$ yields 
\[
\int_{\dom J}e^{-\left(\frac{1}{2t}\left\Vert \bx-\by\right\Vert _{2}^{2}+J(\by)\right)/\epsilon_1}d\by<\int_{\dom J}e^{-\left(\frac{1}{2t}\left\Vert \bx-\by\right\Vert _{2}^{2}+J(\by)\right)/\epsilon_2}\diff\by,
\]
As a result, the function $\epsilon\mapsto\epsilon^{n/2}w_{\epsilon}(\bx,t)$ is strictly monotone increasing on $(0,+\infty)$.

Proof of Lemma~\ref{lem:heat_eqn} (ii)(d): Let $\epsilon>0$, $t>0$, $\lambda\in(0,1)$, $\bx_1,\bx_2\in \Rn$ with $\bx_1\neq\bx_2$ and $\bx=\lambda\bx_{1}+(1-\lambda)\bx_{2}$. Then
\begin{align*}
    e^{\frac{1}{2t\epsilon}\normsq{\bx}}w_\epsilon(\bx,t) &= \frac{1}{(2\pi t\epsilon)^{n/2}}\int_{\dom J} e^{\left(\left\langle \bx,\by\right\rangle/t -\frac{1}{2t}\normsq{\by} - J(\by)\right)/\epsilon}\diff\by \\
    &=\int_{\dom J} \left(\frac{e^{\left(\left\langle \bx_1,\by\right\rangle/t-\frac{1}{2t}\normsq{\by} - J(\by)\right)/\epsilon}}{(2\pi t \epsilon)^{n/2}}\right)^\lambda \left(\frac{e^{\left\langle \bx_2,\by\right\rangle/t\epsilon-\frac{1}{2t}\normsq{\by} - J(\by)/\epsilon}}{(2\pi t \epsilon)^{n/2}}\right)^{1-\lambda} \diff\by.
\end{align*}
H\"older's inequality (\citep{folland2013real}, theorem 6.2) then implies
\begin{align*}
    e^{\frac{1}{2t\epsilon}\normsq{\bx}}w_\epsilon(\bx,t) &\leqslant \left(\int_{\dom J} \frac{e^{\left(\left\langle \bx_1,\by\right\rangle/t-\frac{1}{2t}\normsq{\by} - J(\by)\right)/\epsilon}}{(2\pi t \epsilon)^{n/2}} \diff\by\right)^\lambda \left(\int_{\dom J}  \frac{e^{\left\langle \bx_2,\by\right\rangle/t\epsilon-\frac{1}{2t}\normsq{\by} - J(\by)/\epsilon}}{(2\pi t \epsilon)^{n/2}} \diff\by \right)^{1-\lambda} \\
    &=\left(e^{\frac{1}{2t\epsilon}\normsq{\bx_1}}w_\epsilon(\bx_1,t)\right)^\lambda \left(e^{\frac{1}{2t\epsilon}\normsq{\bx_2}}w_\epsilon(\bx_2,t)\right)^{1-\lambda},
\end{align*}
where the inequality in the equation above is an equality if and only if there exists a constant $\alpha \in \mathbb{R}$ such that $\alpha e^{\left\langle \bx_1,\by\right\rangle/t\epsilon} = e^{\left\langle \bx_x,\by\right\rangle/t\epsilon}$
for almost every $\by \in \dom J$. This does not hold here since $\bx_1\neq\bx_2$. As a result, the function $\Rn\ni\bx\mapsto e^{\frac{1}{2t\epsilon}\normsq{\bx}}w_{\epsilon}(\bx,t)$ is strictly log-convex.
\end{proof}

Proof of Theorem~\ref{thm:visc_hj_eqns} (i) and (ii)(a)-(d): The proof of these follow from Lemma~\ref{lem:heat_eqn} and classic results about the Cole--Hopf transform (see, e.g., \citep{evans1998partial}, Section 4.4.1), with $S_\epsilon(\bx,t) \coloneqq -\epsilon\log(w_\epsilon(\bx,t))$.

Proof of Theorem~\ref{thm:visc_hj_eqns} (iii): The formulas follow from a straightforward calculation of the gradient, divergence, and Laplacian of $S_\epsilon(\bx,t)$ that we omit here. Since the function $\bx \mapsto \frac{1}{2}\normsq{\bx}-tS_{\epsilon}(\bx,t)$ is strictly convex, we can invoke (Corollary 26.3.1, \cite{rockafellar1970convex}) to conclude that its gradient, which is precisely $\bu_{PM}(\bx,t,\epsilon)$, is bijective.

Proof of Theorem~\ref{thm:visc_hj_eqns} (iv): The proof we present here is based on techniques from large deviation theory in probability theory \citep{dembo38large,deuschel2001large,varadhan2016large} tailored to equation \eqref{eq:hj_s_eps} and is adapted from Lemmas 2.1.7 and 2.1.8 of Deuschel and Stroock's book on Large Deviations \citep{deuschel2001large}. We proceed in three steps:
\begin{enumerate}
\item[Step 1.] Show that 
\[
\limsup_{\substack{\epsilon\to0\\
\epsilon>0
}
}S_{\epsilon}(\bx,t)\leqslant\inf_{\by\in\inter(\dom J)}\left\{ \frac{1}{2t}\left\Vert \bx-\by\right\Vert _{2}^{2}+J(\by)\right\} 
\]
 and 
\[
\inf_{\by\in\inter(\dom J)}\left\{ \frac{1}{2t}\left\Vert \bx-\by\right\Vert _{2}^{2}+J(\by)\right\} = \inf_{\by\in\dom J}\left\{ \frac{1}{2t}\left\Vert \bx-\by\right\Vert _{2}^{2}+J(\by)\right\} = S_0(\bx,t).
\]
\item[Step 2.] Show that $\liminf_{\substack{\epsilon\to0\\
\epsilon>0
}
}S_{\epsilon}(\bx,t)\geqslant S_0(\bx,t)$. 
\item[Step 3.] Conclude that $\lim_{\substack{\epsilon\to0\\
\epsilon>0
}
}S_{\epsilon}(\bx,t)=S_0(\bx,t)$. Pointwise and local uniform convergence of the gradient $\lim_{\substack{\epsilon\to0\\
\epsilon>0
}
}\nabla_{\bx}S_{\epsilon}(\bx,t)=\nabla_{\bx}S_0(\bx,t)$, the partial derivative $\lim_{\substack{\epsilon\to0\\
\epsilon>0
}
}\frac{\partial S_{\epsilon}(\bx,t)}{\partial t}=\frac{\partial S_0(\bx,t)}{\partial t}$, and the Laplacian $\lim_{\substack{\epsilon\to0\\
\epsilon>0}}\frac{\epsilon}{2}\Laplacian S_{\epsilon}(\bx,t)=0$ then follow from the convexity and differentiability of the solutions $(\bx,t)\mapsto S_0(\bx,t)$ and $(\bx,t)\mapsto S_{\epsilon}(\bx,t)$ to the HJ PDEs~\eqref{eq:hj_non_visc_pde} and~\eqref{eq:hj_visc_pde}.
\end{enumerate}

We will use the following large deviation principle \citep{deuschel2001large}: For every Lebesgue measurable set $\mathcal{A} \in \Rn$,
\[
\lim_{\substack{\epsilon\to0\\
\epsilon>0
}
}-\epsilon\ln\left(\frac{1}{(2\pi t\epsilon)^{n/2}}\int_{\mathcal{A}}e^{-\frac{1}{2t\epsilon}\left\Vert \bx-\by\right\Vert _{2}^{2}}\diff\by\right)=\essinf_{\by\in\mathcal{A}}\frac{1}{2t}\left\Vert \bx-\by\right\Vert _{2}^{2},
\]
where essential infimum means infimum that holds almost everywhere, that is, if the essential infimum above is attained at $\boldsymbol{a}\in\mathcal{A}$, then $\left\Vert \bx-\boldsymbol{a}\right\Vert _{2}^{2}/2t\leqslant\left\Vert \bx-\by\right\Vert _{2}^{2}/2t$ for almost every $\by\in\mathcal{A}$. 

Step 1. (Adapted from \citet{deuschel2001large}, Lemma 2.1.7.) By convexity, the function $J$ is continuous for every $\by_{0}\in\inter(\dom J)$, the latter set being open. Therefore, for every such $\by_{0}$ there exists a number $r_{\by_{0}}>0$ such that for every $0<r\leqslant r_{\by_{0}}$ the open ball $B_{r}(\by_{0})$ is contained in $\inter(\dom J)$. Hence
\begin{align*}
S_{\epsilon}(\bx,t) & \coloneqq-\epsilon\ln\left(\frac{1}{(2\pi t\epsilon)^{n/2}}\int_{\inter(\dom J)}e^{-(\frac{1}{2t}\normsq{\bx-\by}+J(\by))/\epsilon}\diff\by\right)\\
 & \leqslant-\epsilon\ln\left(\frac{1}{(2\pi t\epsilon)^{n/2}}\int_{B_{r}(\by_{0})}e^{-(\frac{1}{2t}\normsq{\bx-\by}+J(\by))/\epsilon}\diff\by\right)\\
 & \leqslant-\epsilon\ln\left(\frac{1}{(2\pi t\epsilon)^{n/2}}\int_{B_{r}(\by_{0})}e^{-\frac{1}{2t\epsilon}\normsq{\bx-\by}}\diff\by\right)+\sup_{\by\in B_{r}(\by_{0})}J(\by).
\end{align*}
Take $\limsup_{\substack{\epsilon\to0\\ \epsilon>0}}$ and apply the large deviation principle to the term on the right to get
\[
\limsup_{\substack{\epsilon\to0\\\epsilon>0}}S_{\epsilon}(\bx,t)\leqslant\essinf_{\by\in B_{r}(\by_{0})}\frac{1}{2t}\left\Vert \bx-\by\right\Vert _{2}^{2}+\sup_{\by\in B_{r}(\by_{0})}J(\by).
\]
Taking $\lim_{r\to0}$ on both sides of the inequality yields
\[
\limsup_{\substack{\epsilon\to0\\\epsilon>0}}S_{\epsilon}(\bx,t)\leqslant\frac{1}{2t}\left\Vert \bx-\by_{0}\right\Vert _{2}^{2}+J(\by_{0}).
\]
Since the inequality holds for every $\by_{0}\in\inter(\dom J)$, we can take the infimum over all $y\in\inter(\dom J)$ on the right-hand-side of the inequality to get
\[
\limsup_{\substack{\epsilon\to0\\\epsilon>0}}S_{\epsilon}(\bx,t)\leqslant\inf_{\by\in\inter(\dom J)}\left\{ \frac{1}{2t}\left\Vert \bx-\by\right\Vert _{2}^{2}+J(\by)\right\}.
\]
By convexity of $J$, the infimum on the right hand side is equal to that taken over $\dom J$ (\citep{rockafellar1970convex}, Corollary 7.3.2), i.e.,
\begin{equation*}
\inf_{\by\in\inter(\dom J)}\left\{ \frac{1}{2t}\left\Vert \bx-\by\right\Vert _{2}^{2}+J(\by)\right\}   =  \inf_{\by\in\dom J}\left\{ \frac{1}{2t}\left\Vert \bx-\by\right\Vert _{2}^{2}+J(\by)\right\} \equiv S_0(\bx,t).
\end{equation*}
Hence
\[\limsup_{\substack{\epsilon\to0\\\epsilon>0}}S_{\epsilon}(\bx,t)\leqslant S_0(\bx,t).
\]

Step 2. We can directly invoke (\citet{deuschel2001large}, Lemma 2.1.8) to get\footnote{In the notation of \citet{deuschel2001large}, $\Phi=-J$, which is upper semicontinuous, $\by\mapsto\frac{1}{2t}\left\Vert \bx-\by\right\Vert _{2}^{2}$ is the rate function, and note that the tail condition (2.1.9) is satisfied by assumption (A3) in that $\sup_{\by\in\Rn} -J(\by) = -\inf_{\by\in\Rn} J(\by) = 0$.} 
\[
\liminf_{\substack{\epsilon\to0\\\epsilon>0}}S_{\epsilon}(\bx,t)\geqslant S_0(\bx,t).
\]

Step 3. Combining the two limits derived in steps 1 and 2 yield
\[
\lim_{\substack{\epsilon\to0\\\epsilon>0}}S_{\epsilon}(\bx,t)=S_0(\bx,t)
\]
for every $\bx\in\Rn$ and $t>0$, where convergence is uniform on every compact subset $(\bx,t)$ of $\Rn\times(0,+\infty)$ (\citep{rockafellar1970convex}, Theorem 10.8). 

By differentiability and joint convexity of both $(\bx,t) \mapsto S_0(\bx,t)$ and $(\bx,t)\mapsto S_{\epsilon}(\bx,t)-\frac{n\epsilon}{2}\ln t$ (Theorem~\ref{thm:e_u_nonvisc_hj} (i), and Theorem~\ref{thm:visc_hj_eqns} (i) and (ii)(a)), we can invoke (\citep{rockafellar1970convex}, Theorem 25.7) to get 
\[
\lim_{\substack{\epsilon\to0\\\epsilon>0}}\nabla_{\bx}S_{\epsilon}(\bx,t)=\nabla_{\bx}S_0(\bx,t)\,\mbox{and}\,\lim_{\substack{\epsilon\to0\\\epsilon>0}}\left(\frac{\partial S_{\epsilon}(\bx,t)}{\partial t}-\frac{n\epsilon}{2t}\right)=\lim_{\substack{\epsilon\to0\\
\epsilon>0}}\frac{\partial S_{\epsilon}(\bx,t)}{\partial t}=\frac{\partial S_0(\bx,t)}{\partial t},
\]
for every $\bx\in\Rn$ and $t>0$, where convergence is uniform on every compact subset of $\Rn\times(0,+\infty)$. Furthermore, the viscous HJ PDE~\eqref{eq:hj_visc_pde} for $S_\epsilon$ implies that
\begin{align*}
\lim_{\substack{\epsilon\to0\\
\epsilon>0
}
}\frac{\epsilon}{2}\Laplacian S_{\epsilon}(\bx,t) & =\lim_{\substack{\epsilon\to0\\
\epsilon>0
}
}-\left(\frac{\partial S_{\epsilon}(\bx,t)}{\partial t}+\frac{1}{2}\left\Vert \nabla_{\bx}S_{\epsilon}(\bx,t)\right\Vert ^{2}\right),\\
 & =-\left(\frac{\partial S_{0}(\bx,t)}{\partial t}+\frac{1}{2}\left\Vert \nabla_{\bx}S_0(\bx,t)\right\Vert ^{2}\right)\\
 & =0,
\end{align*}
where the last equality holds by the structure of the first-order HJ PDE~\ref{eq:hj_non_visc_pde} (see Theorem \ref{thm:e_u_nonvisc_hj}). Here, again, the limit holds for every $\bx\in\Rn$ and $t>0$, and convergence is uniform over any compact subset of $\Rn\times(0,+\infty)$. Finally, the limit $\lim_{\substack{\epsilon\to0\\\epsilon>0}} \bu_{PM}(\bx,t,\epsilon) = \bu_{MAP}(\bx,t)$ holds directly as a consequence to the limit $\lim_{\substack{\epsilon\to0\\\epsilon>0}}\nabla_{\bx}S_{\epsilon}(\bx,t)=\nabla_{\bx}S_0(\bx,t)$ and the representation formulas~\eqref{eq:grad_s_eps} and \eqref{eq:grad_map_rep} for the posterior mean and MAP estimates, respectively.
\section{Proof of Proposition~\ref{prop:topo_properties}} \label{app:topo}
We will prove that $\bu_{PM}(\bx,t,\epsilon)\in\inter(\dom J)$ in two steps. First, we will use the linearity of the projection operator \eqref{eq:proj_def} and the posterior mean estimate to prove by contradiction that $\bu_{PM}(\bx,t,\epsilon)\in\cl(\dom J)$. Second, we will use the following variant of the Hahn--Banach theorem for convex body in $\Rn$ to show in fact that $\bu_{PM}(\bx,t,\epsilon)\in\inter(\dom J)$.
\begin{thm} \label{thm:tech_res_2} (\citep{rockafellar1970convex}, Theorem 11.6 and Corollary 11.6.2)
Let $C$ be a convex set. A point $\bx\in C$ is a relative boundary point of $C$ if and only if there exist a vector $\boldsymbol{a}\in\Rn\backslash\{\boldsymbol{0}\}$ and a number $b\in\R$ such that 
\[
\bx=\argmax_{\by\in C}\left\langle \boldsymbol{a},\by\right\rangle +b,
\]
i.e., there exists an affine function on $C$ that is not identically constant and achieves its maximum over $C$ at $\bx$.
\end{thm}

Step 1. Suppose $\bu_{PM}(\bx,t,\epsilon)\notin\cl(\dom J)$. Since the set $\cl(\dom J)$ is closed and convex, the projection of $\bu_{PM}(\bx,t,\epsilon)$ onto $\cl(\dom J)$ given by $\pi_{\mbox{cl}(\mbox{dom }J)}(\bu_{PM}(\bx,t,\epsilon)) \equiv \bar{\bu}$ is well-defined and unique (see Definition \ref{def:projections}), with $\bu_{PM}(\bx,t,\epsilon)\neq\bar{\bu}$ by assumption. It also satisfies the characterization \eqref{eq:proj_char}, namely
\[
\left\langle \bu_{PM}(\bx,t,\epsilon)-\bar{\bu},\by-\bar{\bu})\right\rangle \leqslant0
\]
for every $\by\in\cl(\dom J)$. However, by linearity of the posterior mean estimate,
\begin{equation*}
\left\Vert \bu_{PM}(\bx,t,\epsilon)-\bar{\bu}\right\Vert ^{2} = \expectationJ{\left\langle \bu_{PM}(\bx,t,\epsilon)-\bar{\bu},\by-\bar{\bu}\right\rangle)} \leqslant 0,
\end{equation*}
which implies that $\bu_{PM}(\bx,t,\epsilon) = \bar{\bu}$, in contradiction with the assumption that $\bu_{PM}(\bx,t,\epsilon)\notin\cl(\dom J)$.

Step 2. Suppose $\bu_{PM}(\bx,t,\epsilon)\notin\inter(\dom J)$. We know $\bu_{PM}(\bx,t,\epsilon)\in\cl(\dom J)$ by the first step, and therefore $\bu_{PM}(\bx,t,\epsilon)$ must be a boundary point of $\dom J$. If $J$ has no boundary point, then we get a contradiction and conclude $\bu_{PM}(\bx,t,\epsilon)\in\inter(\dom J)$. If not, Theorem \ref{thm:tech_res_2} applies and therefore there exist a vector $\boldsymbol{a}\in\Rn\backslash\{\boldsymbol{0}\}$ and a number $b\in\R$ such that $\bu_{PM}(\bx,t,\epsilon)=\argmax_{_{\by\in\cl(\dom J)}}\left\{\left\langle \boldsymbol{a},\by\right\rangle +b\right\}$, with $\left\langle \boldsymbol{a},\by\right\rangle +b<\left\langle \boldsymbol{a},\bu_{PM}(\bx,t,\epsilon)\right\rangle +b$ for every $\by\in\inter(\dom J)$. However, by linearity of the posterior mean estimate,
\begin{align*}
\left\langle \boldsymbol{a},\bu_{PM}(\bx,t,\epsilon)\right\rangle +b & = \expectationJ{\left\langle \boldsymbol{a},\by\right\rangle +b} \\
& < \expectationJ{\left\langle \boldsymbol{a},\bu_{PM}(\bx,t,\epsilon)\right\rangle +b}\\
& =\left\langle \boldsymbol{a},\bu_{PM}(\bx,t,\epsilon)\right\rangle +b,
\end{align*}
where the strict inequality in the second line follows from integrating over $\inter(\dom J)$. This gives a contradiction, and hence $\bu_{PM}(\bx,t,\epsilon)\in\inter(\dom J)$. 

As a consequence, the subdifferential of $J$ at $\bu_{PM}(\bx,t,\epsilon)$ is non-empty because the subdifferential $\partial J$ is non-empty at every point $\by \in \inter (\dom J)$ (\citep{rockafellar1970convex}, Theorem 23.4). Hence there exists a subgradient $\bp \in \partial J(u_{PM}(\bx,t,\epsilon))$ such that
\begin{equation*}
    J(\by)\geqslant J(\bu_{PM}(\bx,t,\epsilon))-\left\langle \bp,\by-\bu_{PM}(\bx,t,\epsilon)\right\rangle.
\end{equation*}
Taking the expectation $\expectationJ{\cdot}$ on both sides yield the inequality $J(\bu_{PM}(\bx,t,\epsilon)) \leqslant \expectationJ{J(\by)}$. Now, the convex inequality $1+z\leqslant e^{z}$ (which holds for every $z\in\R\cup\{+\infty\}$ with the understanding that $+\infty\leqslant+\infty$) with the choice of $z=J(\by)/\epsilon$ for $\by\in\Rn$ yield $J(\by)e^{-J(\by)/\epsilon}\leqslant\epsilon(1-e^{-J(\by)/\epsilon})$. After multiplying both sides by $\frac{1}{Z_J(\bx,t,\epsilon)}e^{-\frac{1}{2t\epsilon}\left\Vert \bx-\by\right\Vert _{2}^{2}}$ and integrating with respect to \textbf{$\by$}, we find $\expectationJ{J(\by)}\leqslant\epsilon\left(e^{S_\epsilon(\bx,t)/\epsilon}-1\right) < + \infty$.
\section{Proof of Proposition~\ref{prop:pm_rep_props}} \label{app:rep}
Proof of (i): Here, we suppose $\dom J = \Rn$. Let $D_J$ denote the set of points at which $J$ is continuously differentiable on $\Rn$, let $\nabla J(\by)$ denote the gradient of $J$ at these points, let $N_J$ denote the set of the points at which $J$ is not continuously differentiable on $\Rn$, and let $\bv \in \Rn\setminus\{\boldsymbol{0}\}$ by any nonzero vector in $\Rn$. We will derive the representation formulas~\eqref{eq:diff_pm_rep}~and~\eqref{eq:diff_var_rep} by showing that the divergence of the vector fields $\by \mapsto \by e^{-\left(\frac{1}{2t}\left\Vert \bx-\by\right\Vert _{2}^{2} + J(\by)\right)/\epsilon}$ and $\by \mapsto \bv e^{-\left(\frac{1}{2t}\left\Vert \bx-\by\right\Vert _{2}^{2} + J(\by)\right)/\epsilon} $ integrate to zero on $D_J$, i.e.,
\begin{equation}\label{eq:appB_int1}
\int_{D_J} \nabla_{\by} \cdot \left(\bv e^{-\left(\frac{1}{2t}\left\Vert \bx-\by\right\Vert _{2}^{2} + J(\by)\right)\epsilon}\right) \diff{\by} = 0.
\end{equation}
and
\begin{equation}\label{eq:appB_int2}
\int_{D_J} \nabla_{\by} \cdot \left(\by e^{-\left(\frac{1}{2t}\left\Vert \bx-\by\right\Vert _{2}^{2} + J(\by)\right)\epsilon}\right) \diff{\by} = 0.
\end{equation}
We will first assume that these equations hold and derive the representation formulas~\eqref{eq:diff_pm_rep}~and~\eqref{eq:diff_var_rep}, and we will then prove that equations~\eqref{eq:appB_int1} and~\eqref{eq:appB_int2} hold.

Suppose that equations \eqref{eq:appB_int1} and \eqref{eq:appB_int2} hold. First, write the integral in~\eqref{eq:appB_int1} as
\[\left \langle \bv,  \int_{D_J} \left(\frac{\by - \bx}{t} + \nabla J(\by)\right) e^{-\left(\frac{1}{2t}\left\Vert \bx-\by\right\Vert _{2}^{2} + J(\by)\right)/\epsilon} \diff{\by}  \right \rangle = 0.
\]
As $\bv$ is an arbitrary nonzero vector, the integral in the inner product is equal to zero. Since the minimal subgradient $\pi_{\partial J(\by)}(\boldsymbol{0}) = \nabla J(\by)$ everywhere on $D_J$, the set $D_J$ is dense in $\Rn$, and the $n$-dimensional Lebesgue measure of $N_J$ is zero (\cite{rockafellar1970convex}, Theorem 25.5), the gradient $\nabla J(\by)$ in this integral may be replaced with the minimal subgradient $\pi_{\partial J(\by)}(\boldsymbol{0})$ and the domain be taken to be $\Rn$ without changing the value of the integral. Hence, we have
\[
\int_{\Rn} \left(\frac{\by - \bx}{t} + \pi_{\partial J(\by)}(\boldsymbol{0})\right) e^{-\left(\frac{1}{2t}\left\Vert \bx-\by\right\Vert _{2}^{2} + J(\by)\right)/\epsilon} \diff{\by} = 0.
\]
Dividing through by the partition function $Z_J(\bx,t,\epsilon)$ and rearranging yield
\[
\bu_{PM}(\bx,t,\epsilon) = \bx -t\expectationJ{\pi_{\partial J(\by)}(\boldsymbol{0})},
\]
which is the representation formula~\eqref{eq:diff_pm_rep}. In particular, we find the representation formula $\nabla_{\bx}S_\epsilon(\bx,t) = \expectationJ{\pi_{\partial J(\by)}(\boldsymbol{0})}$ via Equation~\eqref{eq:grad_s_eps} in Theorem~\ref{thm:visc_hj_eqns}(iii). Next, write the integral in~\eqref{eq:appB_int2} as
\[\int_{D_J} \left(n\epsilon - \left \langle \by , \left(\frac{\by - \bx}{t} + \nabla J(\by)\right)\right \rangle \right) e^{-\left(\frac{1}{2t}\left\Vert \bx-\by\right\Vert _{2}^{2} + J(\by)\right)/\epsilon}  \diff{\by} = 0.
\]
As discussed previously, we may replace the gradient $\nabla J(\by)$ with the minimal subgradient $\pi_{\partial J(\by)}(\boldsymbol{0})$ and take the domain to be $\Rn$ in this integral without affecting its value. With these changes and on dividing through by the partition function $Z_J(\bx,t,\epsilon)$, we find
\[
\expectationJ{\left \langle \by , \left(\frac{\by - \bx}{t} + \pi_{\partial J(\by)}(\boldsymbol{0})\right)\right \rangle} = n\epsilon.
\]
We can re-write this as
\[
\expectationJ{\left\langle \by, \by-\bx \right\rangle + t\left\langle\bu_{PM}(\bx,t,\epsilon), \pi_{\partial J(\by)}(\boldsymbol{0})\right\rangle} = nt\epsilon -t\expectationJ{\left\langle \by-\bu_{PM}(\bx,t,\epsilon), \pi_{\partial J(\by)}(\boldsymbol{0}) \right\rangle},
\]
and we can write the left hand side as $\expectationJ{\normsq{\by-\bu_{PM}(\bx,t,\epsilon)}}$ using the representation formula $\expectationJ{\pi_{\partial J(\by)}(\boldsymbol{0})} = \left(\frac{\bx-\bu_{PM}(\bx,t,\epsilon)}{t}\right)$ derived previously, which gives the representation formula~\eqref{eq:diff_var_rep}. In particular, we find the representation formula $\Laplacian S_{\epsilon}(\bx,t) = \frac{1}{t\epsilon}\expectationJ{\left \langle \pi_{\partial J(\by)}(\boldsymbol{0}), \by - \bu_{PM}(\bx,t,\epsilon) \right \rangle}$ via Equation~\eqref{eq:exact_variance} in Theorem~\ref{thm:visc_hj_eqns}(iii).

We now establish the two equalities~\eqref{eq:appB_int1} and \eqref{eq:appB_int2}. To do so, we will use a measure-theoretic version of the divergence theorem due to \citet{pfeffer1990divergence} that will apply to the vector fields $\by e^{-\left(\frac{1}{2t\epsilon}\left\Vert \bx-\by\right\Vert _{2}^{2} + J(\by)\right)}$ and $\bv e^{-\left(\frac{1}{2t\epsilon}\left\Vert \bx-\by\right\Vert _{2}^{2} + J(\by)\right)}$.

Let $r>0$ and $\cballr \coloneqq \{\by \in \Rn \colon \left\Vert \by \right\Vert_2 \leqslant r\}$ denote the set of points in the open ball of radius $r$ centered at $\boldsymbol{0}$ in $\Rn$. As $\cballr$ is a bounded, convex, and open set on which the function $J$ is Lipschitz continuous (\cite{rockafellar1970convex}, Theorem 10.4), the set of points $N_J$ at which $J$ is not differentiable in the closed ball $\cl{\cballr}$ constitutes a $\sigma$-finite measurable set with respect to the $n-1$ dimensional Lebesgue measure (\cite{alberti1992singularities}, Theorem 4.1). Thanks to these properties, we can invoke (\cite{pfeffer1990divergence}, Theorem 4.14) to conclude that there exist two unit normal vectors $\boldsymbol{n}_{\by}$ and $\boldsymbol{n}_{\bv}$ such that
\[
\int_{D_J\cap~\cl{\cballr}} \nabla_{\by}\cdot\left(\bv e^{-(\frac{1}{2t}\normsq{\bx - \by} + J(\by))/\epsilon}\right) \diff\by = \int_{\bd\left(D_J\cap~\cl{\cballr}\right)} (\bv\cdot\boldsymbol{n}_{\bv})e^{-\left(\frac{1}{2t}\normsq{\bx - \by} + J(\by)\right)/\epsilon} \diff{\bS}
\]
and
\[
\int_{D_J\cap~\cl{\cballr}} \nabla_{\by}\cdot\left(\by e^{-(\frac{1}{2t}\normsq{\bx - \by} + J(\by))/\epsilon}\right) \diff\by = \int_{\bd\left(D_J\cap~\cl{\cballr}\right)} (\by\cdot\boldsymbol{n}_{\by})e^{-\left(\frac{1}{2t}\normsq{\bx - \by} + J(\by)\right)/\epsilon} \diff{\bS},
\]
where $\diff{S}$ denotes the $n-1$ dimensional Lebesgue measure. Take the limit $r \to +\infty$ on both sides to get
\[
\int_{D_J} \nabla_{\by}\cdot\left(\bv e^{-(\frac{1}{2t}\normsq{\bx - \by} + J(\by))/\epsilon}\right) \diff\by = \lim_{r\to+\infty} \int_{\bd\left(D_J\cap~\cl{\cballr}\right)} (\bv\cdot\boldsymbol{n_{\bv}})e^{-\left(\frac{1}{2t}\normsq{\bx - \by} + J(\by)\right)/\epsilon} \diff{\bS}
\]
and
\[
\int_{D_J} \nabla_{\by}\cdot\left(\by e^{-(\frac{1}{2t}\normsq{\bx - \by} + J(\by))/\epsilon}\right) \diff\by = \lim_{r\to+\infty}\int_{\bd\left(D_J\cap~\cl{\cballr}\right)} (\by\cdot\boldsymbol{n_{\by}})e^{-\left(\frac{1}{2t}\normsq{\bx - \by} + J(\by)\right)/\epsilon} \diff{\bS},
\]
The right hand side of these equations are equal to zero since the limits
\begin{equation*}
0 \leqslant \lim_{\left\Vert \by\right\Vert _{2}\to+\infty}\left\Vert (\bv\cdot \boldsymbol{n_{\bv}})e^{-\left(\frac{1}{2t}\left \Vert \bx - \by \right \Vert_{2}^{2} + J(\by)\right)/\epsilon}\right\Vert _{2} \leqslant \lim_{\left\Vert \by\right\Vert _{2}\to+\infty}\left\Vert \bv\right\Vert_{2} e^{-\left(\frac{1}{2t}\left \Vert \bx - \by \right \Vert_{2}^{2} \right)/\epsilon} = 0
\end{equation*}
and
\begin{equation*}
0 \leqslant \lim_{\left\Vert \by\right\Vert _{2}\to+\infty}\left\Vert (\by\cdot \boldsymbol{n_{\by}})e^{-\left(\frac{1}{2t}\left \Vert \bx - \by \right \Vert_{2}^{2} + J(\by)\right)/\epsilon}\right\Vert _{2} \leqslant \lim_{\left\Vert \by\right\Vert _{2}\to+\infty}\left\Vert \by\right\Vert_{2} e^{-\left(\frac{1}{2t}\left \Vert \bx - \by \right \Vert_{2}^{2} \right)/\epsilon} = 0
\end{equation*}
and a short calculation yield \[
\lim_{r\to+\infty} \int_{\bd(D_J\cap~\cl{\cballr})} (\bv\cdot\boldsymbol{n_{\bv}})e^{-\left(\frac{1}{2t}\left \Vert \bx - \by \right \Vert_{2}^{2} + J(\by)\right)/\epsilon}) \diff \bS= 0
\] 
and 
\[
\lim_{r\to+\infty} \int_{\bd\left(D_J\cap~\cl{\cballr}\right)} (\bv\cdot\boldsymbol{n_{\bv}})e^{-\left(\frac{1}{2t}\normsq{\bx - \by} + J(\by)\right)/\epsilon} \diff{\bS}.
\]
Hence Equations~\eqref{eq:appB_int1} and \eqref{eq:appB_int2}, and the proof of (i) is finished.

Proof of (ii): Let $\{\mu_k\}_{k=1}^{+\infty}$ be a sequence of positive real numbers converging to zero and $\by \in \dom J$. By Theorem~\ref{thm:e_u_nonvisc_hj}(i), the sequence of real numbers $\{S_0(\bx,\mu_k)\}_{k=1}^{+\infty}$ converges to $J(\by)$, and by assumption (A3) that $\inf_{\by \in \Rn} J(\by) = 0$ the sequence $\{S_0(\bx,\mu_k)\}_{k=1}^{+\infty}$ is bounded uniformly from below by $0$, i.e.,
\begin{align*}
S_0(\bx,\mu_k) & =\inf_{\by\in\Rn}\left\{\frac{1}{2\mu_k}\normsq{\bx-\by} + J(\by)\right\} \\
 & \geqslant\inf_{\by\in\Rn}\left\{\frac{1}{2\mu_k}\normsq{\bx-\by}\right\} +\inf_{\by\in\Rn}J(\by)\\
 & = 0,
\end{align*}
and therefore,
\begin{equation*}
\frac{1}{(2\pi t\epsilon)^{n/2}}\int_{\Rn}e^{-\left(\frac{1}{2t}\left\Vert \bx-\by\right\Vert _{2}^{2}+S_0(\bx,\mu_k)\right)/\epsilon}\diff\by \leqslant\frac{1}{(2\pi t\epsilon)^{n/2}}\int_{\Rn}e^{-\frac{1}{2t\epsilon}\left\Vert \bx-\by\right\Vert _{2}^{2}}\diff\by=1.
\end{equation*}
Using the Lebesgue dominated convergence theorem (\citep{folland2013real}, Theorem 2.24) and the limit result $\lim_{k\to+\infty}e^{-S_0(\bx,\mu_k)/\epsilon}=e^{-J(\by)/\epsilon}$ for every $\by\in\dom J$ from Theorem \ref{thm:e_u_nonvisc_hj}(i) (with $\lim_{k\to+\infty}e^{-S_0(\bx,\mu_k)/\epsilon}=0$ for every $\by\notin\dom J$), we find
\[
\lim_{k\to+\infty} -\epsilon\log\left(\frac{1}{(2\pi t\epsilon)^{n/2}}\int_{\Rn}e^{-\left(\frac{1}{2t}\left\Vert \bx-\by\right\Vert _{2}^{2}+S_0(\bx,\mu_k)\right)/\epsilon}\diff\by\right) = S_\epsilon(\bx,t).
\]
Thanks to the convexity of $\bx \mapsto S_\epsilon(\bx,t)$ proven in Theorem~\ref{thm:visc_hj_eqns}(ii)(a) as well as the representation formula~\eqref{eq:diff_pm_rep} derived in part (i) of this proof, we can invoke (\cite{rockafellar1970convex}, Theorem 25.7) to find
\[
\nabla_{\bx}S_\epsilon(\bx,t) = \lim_{k\to +\infty}\left(\frac{\int_{\Rn} \nabla_{\by}S_0(\by,\mu_k)e^{-\left(\frac{1}{2t}\normsq{\bx-\by} + S_0(\by,\mu_k)\right)/\epsilon}\diff\by}{\int_{\Rn} e^{-\left(\frac{1}{2t}\normsq{\bx-\by} + S_0(\by,\mu_k)\right)/\epsilon}\diff\by}\right).
\]
Finally, we can use formula~\eqref{eq:grad_s_eps} and the limit above to find
\[
\bu_{PM}(\bx,t,\epsilon) = \bx - t\lim_{k\to +\infty}\left(\frac{\int_{\Rn} \nabla_{\by}S_0(\by,\mu_k)e^{-\left(\frac{1}{2t}\normsq{\bx-\by} + S_0(\by,\mu_k)\right)/\epsilon}\diff\by}{\int_{\Rn} e^{-\left(\frac{1}{2t}\normsq{\bx-\by} + S_0(\by,\mu_k)\right)/\epsilon}\diff\by}\right).
\]
\section{Proof of Proposition~\ref{prop:mono_properties}} \label{app:mono}
Let $\{\mu_k\}_{k=1}^{+\infty}$ be a sequence of positive real numbers converging to zero and let $S_0\colon \Rn \times (0,+\infty) \to \R$ denote the solution to the first-order HJ PDE~\eqref{eq:hj_non_visc_pde} with initial data $J$. Define the function $F\colon\dom \partial J \times\dom \partial J\times\Rn\times(0,+\infty)\to[0,+\infty)$ as
\begin{equation*}
F(\by,\by_0,\bx,t)=
\left\langle \left(\frac{\by - \bx}{t} + \pi_{\partial J(\by)}(\boldsymbol{0})\right)-\left(\frac{\by_0 - \bx}{t} + \pi_{\partial J(\by_0)}(\boldsymbol{0})\right),\by-\by_{0}\right\rangle e^{-\left(\frac{1}{2t}\left \Vert \bx - \by \right \Vert_{2}^{2} + J(\by)\right)/\epsilon}
\end{equation*}
and the sequence of functions $\{F_{\mu_k}\}_{k=1}^{+\infty}$ with $F_{\mu_k}\colon\dom \partial J\times\dom \partial J\times\Rn\times(0,+\infty)\to[0,+\infty)$ as 
\[
F_{\mu_k}(\by,\by_0,\bx,t)=\left\langle \left(\frac{\by - \bx}{t} + \nabla_{\by}S_0(\by,\mu_k)\right)-\left(\frac{\by_0 - \bx}{t} + \nabla_{\by}S_0(\by_0,\mu_k)\right),\by-\by_{0}\right\rangle e^{-\left(\frac{1}{2t}\left \Vert \bx - \by \right \Vert_{2}^{2} + S_0(\by,\mu_k)\right)/\epsilon}.
\]
Since $\lim_{k\to +\infty} S_0(\by,\mu_k) = J(\by)$ and $\lim_{k\to +\infty} \nabla_{\by}S_0(\by,\mu_k) = \pi_{\partial J(\by)}(\boldsymbol{0})$ for every $\by \in \Rn$ by Theorem~\ref{thm:e_u_nonvisc_hj}(i) and (iv), the limit $\lim_{k \to +\infty} F_{\mu}(\by,\by_0,\bx,t) = F(\by,\by_0,\bx,t)$ holds for every $\by \in \dom\partial J$, $\by_0 \in \dom\partial J$, $\bx \in \Rn$ and $t > 0$. Moreover, the function $F$ and sequence of functions $\{F_{\mu_k}\}_{k=1}^{+\infty}$ are positive functions in their arguments because the strong convexity of both $\dom\partial J \ni \by\mapsto\frac{1}{2t}\left\Vert \bx-\by\right\Vert _{2}^{2}+J(\by)$ and $\dom\partial J \ni \by\mapsto\frac{1}{2t}\left\Vert \bx-\by\right\Vert _{2}^{2}+S_0(\by,\mu)$ with modulus $\left(\frac{1+mt}{t}\right)$ imply
\begin{equation}\label{eq:app_convex_ineq}
\left(\frac{1+mt}{t}\right)\normsq{\by-\by_{0}} \leqslant \left\langle \left(\frac{\by - \bx}{t} + \pi_{\partial J(\by)}(\boldsymbol{0})\right)-\left(\frac{\by_0 - \bx}{t} + \pi_{\partial J(\by_0)}(\boldsymbol{0})\right),\by-\by_{0}\right\rangle
\end{equation}
and
\[
\left(\frac{1+mt}{t}\right)\normsq{\by-\by_{0}} \leqslant \left\langle \left(\frac{\by - \bx}{t} + \nabla_{\by}S_0(\by,\mu_k)\right)-\left(\frac{\by_0 - \bx}{t} + \nabla_{\by}S_0(\by_0,\mu_k)\right),\by-\by_{0}\right\rangle.
\]
As a consequence, Fatou's lemma applies to the sequence of functions $\{F_{\mu_k}\}_{k=1}^{+\infty}$, and hence
\begin{equation*}
\begin{alignedat}{1}
    \int_{\dom \partial J}F(\by,\by_0,\bx,t) \diff\by &\leqslant \liminf_{k \to +\infty} \int_{\dom\partial J} F_{\mu_k}(\by,\by_0,\bx,t) \diff\by \\
    & \leqslant \liminf_{k \to +\infty} \int_{\Rn} F_{\mu_k}(\by,\by_0,\bx,t) \diff\by. \\
\end{alignedat}
\end{equation*}
By the representation formulas~\eqref{eq:diff_pm_rep} and \eqref{eq:diff_var_rep} derived in Proposition~\ref{prop:pm_rep_props} applied to the initial data $\by \mapsto S_0(\by,\mu)$,
\[
  \int_{\Rn} F_{\mu_k}(\by,\by_0\bx,t) \diff\by = \int_{\Rn} \left(n\epsilon - \left\langle \frac{\by_0 -\bx}{t} + \nabla_{\by}S_0(\by_0,\mu_k) , \by-\by_0 \right\rangle\right)e^{-(\frac{1}{2t}\normsq{\bx - \by} + S_0(\by,\mu_k))/\epsilon} \diff\by.
\]
A straightforward calculation using assumption (A3) yields $S_0(\by,\mu_k) \geqslant 0$ for every $k\in\mathbb{N}$, and it implies $\left\Vert \by-\by_0 \right\Vert_{2}e^{-(\frac{1}{2t}\normsq{\bx - \by} + S_0(\by,\mu_k))/\epsilon} \leqslant \left\Vert \by-\by_0 \right\Vert_{2} e^{-(\frac{1}{2t}\normsq{\bx - \by})/\epsilon}$, which is integrable over $\Rn$. Hence the Lebesgue dominated convergence theorem applies and since $\lim_{k \to +\infty} \nabla_{\by}S_0(\by,\mu_k) = \pi_{\partial J(\by)}(\boldsymbol{0})$ by Theorem~\ref{thm:e_u_nonvisc_hj}(iv), we get
\[
\liminf_{k \to +\infty} \int_{\Rn} F_{\mu_k}(\by,\bx,t) \diff\by = n\epsilon - \left\langle \frac{\by_0 -\bx}{t} + \pi_{\partial J(\by_0)}(\boldsymbol{0}), \int_{\dom\partial J}(\by-\by_0)e^{-(\frac{1}{2t}\normsq{\bx - \by} + J(\by))/\epsilon} \diff\by \right\rangle.
\]
Finally, combining this limit with the strong convexity inequality \eqref{eq:app_convex_ineq} and dividing through by the partition function $Z_J(\bx,t,\epsilon)$, we get inequality~\eqref{eq:monotonicity_prop}, which proves Proposition~\ref{prop:mono_properties}.
\section{Proof of Theorem~\ref{thm:bregman_div}}
Proof of (i): For every $\bu \in \Rn$, the Bregman divergence of $\dom\partial J\ni\by\mapsto\Phi_{J}(\by,\bx,t)$ at $(\bu,\varphi_{J}(\by,\bx,t))$ is given by
\begin{align*}
D_{\Phi_{J}}(\bu,\varphi_{J}(\by,\bx,t)) & =\Phi_{J}(\bu,\bx,t)-\left\langle \varphi_{J}(\by,\bx,t),\bu\right\rangle +\Phi_{J}^{*}(\varphi_{J}(\by,\bx,t))\\
 & =\Phi_{J}(\bu,\bx,t)-\Phi_{J}(\by,\bx,t)-\left\langle \varphi_{J}(\by,\bx,t),\bu-\by\right\rangle
\end{align*}
by definition of the subdifferential (see definition \ref{def:legendre_t}), where equality holds because $\varphi_J(\by,\bx,t) \in \partial \Phi_J(\by,\bx,t)$. Note that the expected value $\expectationJ{D_{\Phi_{J}}(\bu,\varphi_{J}(\by,\bx,t))}$ is finite because the expected value $\expectationJ{\left\langle \varphi_J(\by,\bx,t), \bu-\by \right\rangle}$ in the second line of the previous equation is finite thanks to the monotonicity property~\eqref{eq:monotonicity_prop} and finiteness of the mean minimal subgradient $\expectationJ{\pi_{\partial J(\by)}(\boldsymbol{0})}$ by Corollary~\ref{cor:mmsp}. Now, we have
\begin{equation*}
\begin{alignedat}{1}
\expectationJ{D_{\Phi_{J}}(\bu,\varphi_{J}(\by,\bx,t))} &= \Phi_{J}(\bu,\bx,t)-\left\langle \expectationJ{\varphi_{J}(\by,\bx,t)},\bu\right\rangle + \expectationJ{\Phi_{J}^{*}(\varphi_{J}(\by,\bx,t))}
\\
&= \Phi_{J}(\bu,\bx,t)+ \left\langle \nabla_{\bx}S_\epsilon(\bx,t) - \expectationJ{\pi_{\partial J(\by)}(\boldsymbol{0})},\bu\right\rangle + \expectationJ{\Phi_{J}^{*}(\varphi_{J}(\by,\bx,t))}
\\
&= \frac{1}{2t}\normsq{\bx-\bu} + \left(J(\bu) + \left\langle \nabla_{\bx}S_\epsilon(\bx,t) - \expectationJ{\pi_{\partial J(\by)}(\boldsymbol{0})},\bu\right\rangle\right) + \expectationJ{\Phi_{J}^{*}(\varphi_{J}(\by,\bx,t))}.
\end{alignedat}
\end{equation*}
Since $\bu \mapsto J(\bu) + \left\langle \nabla_{\bx}S_\epsilon(\bx,t) - \expectationJ{\pi_{\partial J(\by)}(\boldsymbol{0})},\bu\right\rangle$ is a convex function and $J \in \gmRn$ by assumption (A1), we can invoke Theorem~\ref{thm:e_u_nonvisc_hj}(ii) to conclude that $\bu \mapsto \expectationJ{D_{\Phi_{J}}(\bu,\varphi_{J}(\by,\bx,t))}$ has a unique minimizer $\bar{\bu}$ that satisfies the inclusion
\[
\left(\frac{\bx - \bar{\bu}}{t}\right) \in \partial J(\bar{\bu}) + \left(\nabla_{\bx}S_\epsilon(\bx,t) - \expectationJ{\pi_{\partial J(\by)}(\boldsymbol{0})}\right).
\].

Proof of (ii): If $\dom J = \Rn$, then the representation formula $\nabla_{\bx}S_\epsilon(\bx,t) = \expectationJ{\pi_{\partial J(\by)}(\boldsymbol{0})}$ derived in Proposition~\ref{prop:pm_rep_props} holds and the characterization of the minimizer $\bar{\bu}$ in equation~\eqref{eq:subdiff_rep1} reduces to $\left(\frac{\bx - \bar{\bu}}{t}\right)\in \partial J(\bar{\bu})$. The unique minimizer that satisfies this characterization is the MAP estimate $\bu_{MAP}(\bx,t)$ (Theorem~\ref{thm:e_u_nonvisc_hj}(ii)), i.e., $\bar{\bu} = \bu_{MAP}(\bx,t)$.

\clearpage{}
\bibliographystyle{plainnat}

\end{document}